\documentclass[12pt, oneside,reqno]{amsart}   	
\usepackage{geometry}
\usepackage{float}
\allowdisplaybreaks
\usepackage{graphicx}				
\usepackage{amssymb}
\usepackage{amsmath}
\usepackage[all]{xy}
\usepackage{mathtools}
\usepackage{tikz-cd}
\usepackage{enumitem}
\usepackage{tikz-cd}
\usepackage{appendix}
\usepackage{amsthm}

\newtheorem{rmk}{Remark}[section]
\newtheorem{thm}{Theorem}[section]
\newtheorem{lem}[thm]{Lemma}
\newtheorem{cor}[thm]{Corollary}
\newtheorem{prop}[thm]{Proposition}

\newtheorem{Question}[thm]{Question}

\newtheorem{defn}{Definition}[section]

\numberwithin{equation}{section}

\def\Pb{\ifmmode{\Bbb P}\else{$\Bbb P$}\fi}
\def\Z{\ifmmode{\Bbb Z}\else{$\Bbb Z$}\fi}
\def\C{\ifmmode{\Bbb C}\else{$\Bbb C$}\fi}
\def\R{\ifmmode{\Bbb R}\else{$\Bbb R$}\fi}
\def\S{\ifmmode{S^2}\else{$S^2$}\fi}

\newtheorem{innercustomgeneric}{\customgenericname}
\providecommand{\customgenericname}{}
\newcommand{\newcustomtheorem}[2]{%
  \newenvironment{#1}[1]
  {%
   \renewcommand\customgenericname{#2}%
   \renewcommand\theinnercustomgeneric{##1}%
   \innercustomgeneric
  }
  {\endinnercustomgeneric}
}

\newcustomtheorem{customthm}{Theorem}

\def\diam{\operatorname{diam}}
\def\Area{\operatorname{Area}}

\def\S{\cal S}

\begin{document}

\title[Nonconvex Surfaces]{Nonconvex Surfaces which Flow to Round Points}
\author{Alexander Mramor and Alec Payne}
\address{Department of Mathematics, University of California Irvine, Irvine, CA 92617}
\address{Courant Institute, New York University, New York City, NY 10012}
\email{mramora@uci.edu, ajp697@nyu.edu}
\begin{abstract} In this article, we extend Huisken's theorem that convex surfaces flow to round points by mean curvature flow. We construct certain classes of mean convex and non-mean convex hypersurfaces that shrink to round points and use these constructions to create pathological examples of flows. We find a sequence of flows that exist on a uniform time interval, have uniformly bounded diameter, and shrink to round points, yet the sequence of initial surfaces has no subsequence converging in the Gromov-Hausdorff sense. Moreover, we find a sequence of flows which all shrink to round points, yet the initial surfaces converge to a space-filling surface. Also constructed are surfaces of arbitrarily large area which are close in Hausdorff distance to the round sphere yet shrink to round points.
\end{abstract}
\maketitle
\section{Introduction}
In his foundational paper~\cite{Huisken84}, Huisken showed that the mean curvature flow of a convex surface shrinks smoothly to a point and approaches a round sphere after rescaling. In other words, it flows to a round point. On the other hand, Angenent \cite{Angenent92} and Topping~\cite{Topping98} independently showed that neckpinch singularities occur quite generally, meaning that a singularity occurs before the flow contracts to a point. In fact, Angenent's self-shrinking torus constructed in that same paper shows that hypersurfaces that flow to points need not become round at the singular time. Despite the possibility of neckpinch singularities and the possibility of a flow disappearing in a non-round point, it remains a natural question to ask what classes of nonconvex surfaces shrink to round points.
$\medskip$

Progress has been made towards extending Huisken's theorem in terms of curvature pinching conditions, including in higher codimension and in non-Euclidean target spaces---see the works of Andrews-Baker~\cite{AndrewsBaker10}, Liu-Xu-Ye-Zhao~\cite{LiuXuYeZhao18}, and Lei-Xu~\cite{LeiXu, LeiXu}. A result of Lin \cite{LinStarShaped20} gives that surfaces which have very small $L^2$-norm of tracefree second fundamental form will shrink to round points under the flow. In this paper, we construct classes of surfaces that have very large tracefree second fundamental form, including in the $L^2$ sense. From another perspective, Colding-Minicozzi showed that a surface which shrinks to a compact point does so generically to a round point, in the sense of their pieceise mean curvature flow~\cite{ColdingMinicozziGeneric}. Later, Bernstein and L. Wang showed that any surface in $\mathbb{R}^3$ with entropy less than that of the cylinder flows to a round point (see \cite[Corollary 1.2]{BernsteinWang17}). Note that the papers of Lin-Sesum and Colding-Minicozzi make no assumptions on the pointwise geometry of the surfaces involved.
\medskip
\vspace{-.15in}

Our first theorem, Theorem \ref{first theorem}, is an extension of Huisken's theorem to certain nonconvex tubular neighborhoods of curve segments. This is a precursor to the method used to show the next result. We will do this by constructing appropriate inner and outer barriers which, while they will not be mean curvature flows outright, will be subsolutions and supersolutions to the flow. In this theorem and the following, $n \geq 2$ and $A$ denotes the second fundamental form. Also, all hypersurfaces will be smooth unless mentioned otherwise.

\begin{thm}\label{first theorem} Let $\Sigma = \Sigma(n, L)$ denote the space of embedded intervals in $\mathbb{R}^{n+1}$ with length bounded by $L$. Let $\gamma\in \Sigma(n,L)$ be a curve, and let $T_r(\gamma)$ be the topological boundary of the solid tubular neighborhood of $\gamma$, using the Euclidean distance in $\mathbb{R}^{n+1}$. Then for every $L > 0$, there exists $C>0$ with the following significance: For each $\gamma \in \Sigma(n,L)$ with $|A|, |\nabla A| \leq C$, there exists $0<r<L/10$ such that for all $\epsilon \ll 1$, there exists a smooth embedded hypersurface which is $\epsilon$-close in $C^1$ to $T_r(\gamma)$ and which shrinks by mean curvature flow to a round point in finite time. Moreover, there is a lower bound $\overline{C} < C$ which depends only on $n$, $L$, and $\epsilon_0$, where $\epsilon_0$ is the constant from the Brakke regularity theorem  (see Section 2, Theorem \ref{Brakke}). 
\end{thm} 

In Theorem \ref{first theorem}, $\epsilon$ and $r$ are chosen small enough depending on $\gamma$. Also, note that the curvature bound $C$ is chosen small depending on $L$, meaning that the tubular neighborhood of each admissible $\gamma$ is embedded. Also, if one could concretely estimate the constant $\epsilon_0$, then one could in theory explicitly compute $\overline{C}$ in this theorem.  Building off of this, we will show how one can add ``spikes'' to the examples given by Theorem \ref{first theorem} to give new examples of non-mean convex, arbitrarily high curvature hypersurfaces which shrink to round points. By ``spikes,'' we mean a perturbation of a thin tubular neighborhood of a curve segment with one endpoint attached orthogonally to a surface. 
$\medskip$

Before we state the next theorem, we must first explain some notation. Let $\overline{\Sigma}$ be the set of smooth hypersurfaces which shrink to round points, endowed with the $C^2$ topology. The set $\overline{\Sigma}$ contains $C^2$ small perturbations of strictly convex surfaces, as well as $C^2$ small perturbations of the surfaces in Theorem \ref{first theorem}. Indeed, $\overline{\Sigma}$ is open in the $C^2$ topology. That is, for a smooth surface $M$ which flows to a round point, there exists $\delta>0$ such that for all perturbations $M^*$ of $M$ that are $\delta$-close in $C^2$ to $M$, $M^*$ flows to a round pount. This follows from continuous dependence of the flow under smooth perturbations, combined with derivative estimates for the flow. See the appendix, Theorem \ref{appendix theorem 1}. Also, in the following theorem, we use a slightly different definition of $T_r(\gamma)$ than in Theorem \ref{first theorem}, in order to have more concision with the notation. We will make our notation clear throughout our arguments.

\begin{thm}\label{second theorem} Let $M$ be a hypersurface  in the set $\overline{\Sigma}$ defined above. Fix $N \gg 1$. Then, for any $p \in M$ and any $L>0$, there exists $0<r\ll \epsilon \ll L$ such that for any straight line segment $\gamma$ orthogonal to $T_p M$ with an endpoint at $p$, there exists a closed (possibly immersed) hypersurface $\widetilde{M}$ with the following properties:
\begin{enumerate}
    \item The flow $\widetilde{M}_t$ shrinks to a round point, i.e.\ $\widetilde{M} \in \overline{\Sigma}$,
    \item $\widetilde{M} \cap T_{r}(\gamma)$ is given by a graph over $T_p M \cap B(p,\epsilon)$ and the mean curvature of $\widetilde{M}\cap T_{\epsilon}(\gamma)$ has a sign\footnote{By $\widetilde{M} \cap T_{\epsilon}(\gamma)$, we mean the connected component (of the preimage of the natural immersion defining $\widetilde{M}$) of $\widetilde{M} \cap T_{\epsilon}(\gamma)$ containing the glued-in $\partial T_r(\gamma)$. This accounts for the possibility that there are other parts of $\widetilde{M}$, as subsets of Euclidean space, which intersect $T_{\epsilon}(\gamma)$ and are not part of the ``spike'' we construct.}
    \item $\widetilde{M} = M\cup \partial T_r(\gamma)$ outside the ball $B(p,\epsilon)$,\footnote{See the above footnote. By this, we mean that $\widetilde{M}$ will coincide with $M \cup \partial T_r(\gamma)$ as subsets.}
\end{enumerate}

where $T_{\epsilon}(\gamma)$ is the solid tubular neighborhood of radius $\epsilon$ around $\gamma$ and $\partial T_r(\gamma)$ denotes a smooth surface that is $\frac{r}{N}$-close in $C^1$ to the boundary of the tubular $r$-neighborhood of $\gamma$.

Furthermore, we may iterate this construction by starting with $\widetilde{M} \in \overline{\Sigma}$, as opposed to $M$, and applying the above procedure to some other choice of $p' \in \widetilde{M}$ and $L'>0$.  
\end{thm}

\begin{figure}
\centering
\includegraphics[scale = .4]{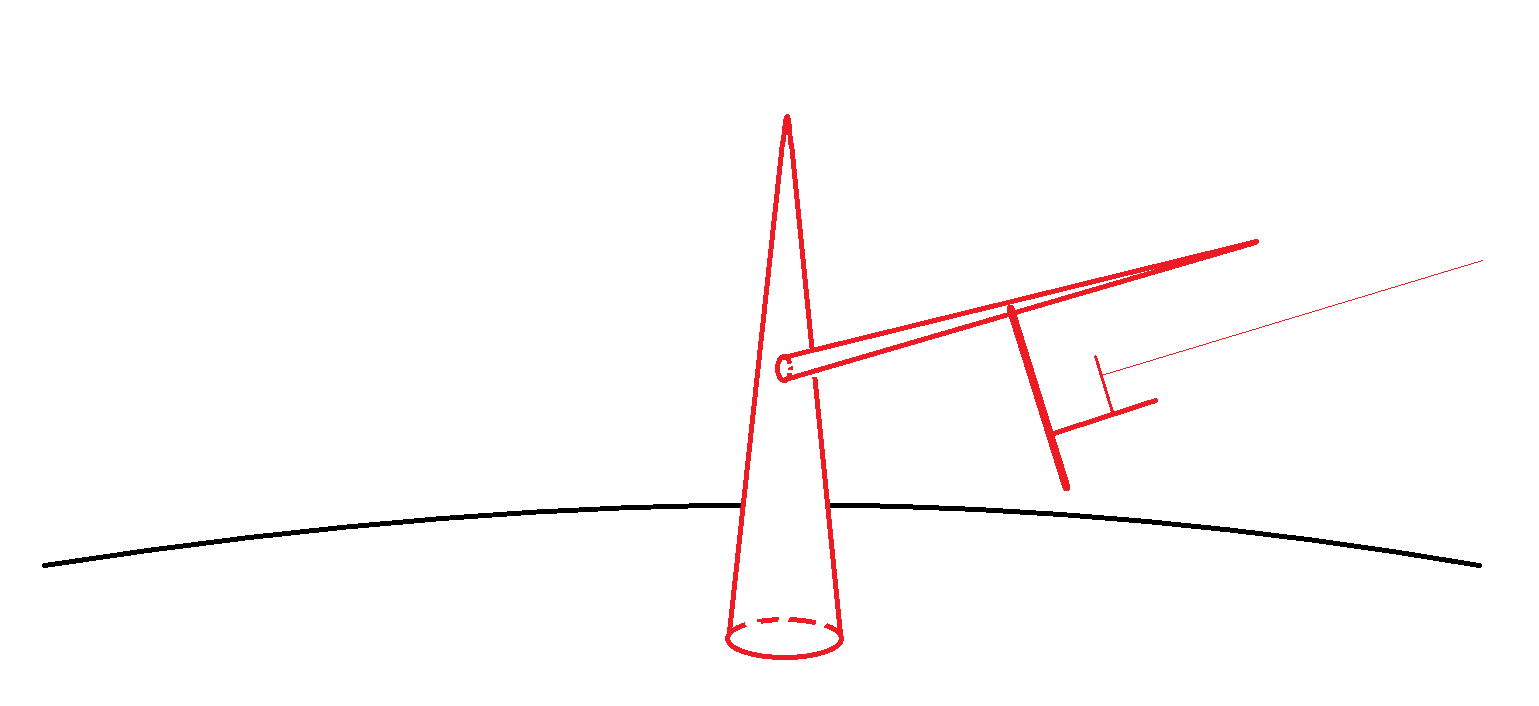}
\caption{We may iterate Theorem \ref{second theorem} to find complicated non-mean convex surfaces which shrink to round points.}
\label{iteration figure}
\end{figure}

Compare with \cite{Mramor} concerning the level set flow, where criterion for level set flows possibly initially far from planes to smoothen are given. The parameter $N$ gives extra flexibility of the construction and allows for further applications when the construction is iterated. This will become apparent in Section 5, where it will be used in the proofs of Corollaries \ref{no GH limit with area unbounded} and \ref{highentropy}.
$\medskip$

Note that we may take $\gamma$ in Theorem \ref{second theorem} to be slightly bent, in analogy with Theorem \ref{first theorem}. That is, for each $L$ and $p$ in the above theorem, we may find a $C$ small enough such that the above theorem holds for $\gamma$ with $|A|^2<C$. Also, the scales $\epsilon$, $r$ at which the ``spikes'' are added may be very small compared to the initial hypersurface.
$\medskip$

We also emphasize that we may construct the spikes in Theorem \ref{second theorem} so that they are either inward-pointing or outward-pointing, i.e. that $\gamma$ merely needs to be orthogonal to $T_p M$ yet this can be either coincident with an inward-pointing normal vector or an outward-pointing normal vector. This means that this theorem gives examples of non-mean convex surfaces which shrink to round points. 
$\medskip$

Our next construction is a natural generalization of Theorem \ref{second theorem} and will be used to construct high area examples. The idea is to add several thin ancient pancakes, which will all shrink in quickly like a spike yet will each contribute a definite amount of area to the surface. By ``ancient pancakes,'' we mean the convex, $O(n)\times O(1)$-invariant ancient solution constructed by Wang and later in more detail by Bourni-Langford-Tinaglia~\cite{Wa11, BLT1}.
$\medskip$

Let $M$ be a rotationally symmetric smooth closed hypersurface in $\overline{\Sigma}$, i.e.\ $M_t$ flows to a round point. By rotationally symmetric, we mean that $M$ can be written as the rotation of a graph $f$ about the $x_{n+1}$-axis. Fixing the interval $[a,b]\subset \mathbb{R}$, we say that $M$ is $(\delta,c)$-cylindrical over $[a,b]$ if $M$ is $\delta$-close in $C^2$ to a segment of the standard cylinder $S^{n-1} \times \mathbb{R}$ of radius $c$ on the interval $[a,b]$, i.e.\ if $f$ is close to the constant function $c$ over $[a,b]$.

\begin{thm}\label{third theorem}  Let $M$ be a smooth closed hypersurface in $\overline{\Sigma}$ which is rotationally symmetric, i.e. it can be written as the rotation of a graph $f$ about the $x_{n+1}$-axis. Let $p \in (a,b)$ and $L>0$. If $M$ is $(\delta,c)$-cylindrical over $[a,b]$ for $\delta \ll c$, then there exists $r \ll \epsilon \ll 1$, a closed hypersurface $\widetilde{M}$, and a smooth positive graph $\widetilde{f}:[a,b] \to \mathbb{R}_{\geq 0}$ with the following properties:
\begin{enumerate}
    \item $\widetilde{M}_t$ flows to a round point, i.e.\ $\widetilde{M} \in \overline{\Sigma}$,
    \item $\widetilde{M}$ is rotationally symmetric and can be represented by the rotation of $\widetilde{f}$ around the $x_{n+1}$-axis,
    \item $\widetilde{f}\geq f$ and $f= \widetilde{f}$ outside $[p-\epsilon, p+\epsilon]$,
    \item $\widetilde{f}(p) = f(p)+L$,
    \item Outside the ball $B(f(p), \epsilon)$, $\widetilde{f}$ is $\frac{r}{100}$-close in $C^1$ to the boundary of the $r$-neighborhood of the set $\{t \in [0,L] \,|\,(p,f(p)+t(L-r))\}$.
\end{enumerate}
\end{thm}

Now we state several corollaries of Theorems \ref{second theorem} and \ref{third theorem} which will be proven in Section 6. The first corollary, a consequence of Theorem \ref{second theorem}, will demonstrate how badly compactness of mean curvature flows can fail without a uniform bound on the second fundamental form. The following corollary is summarized in Figure \ref{spiky ball figure}.

\begin{cor}\label{no GH limit corollary}
There exists a sequence of closed hypersurfaces $M^i\subset \mathbb{R}^{n+1}$, such that $\diam(M^i)$ and $\Area(M^i)$ are uniformly bounded and each flow $M^i_t$ exists on a uniform time interval $[0,T]$ and shrinks to a round point, yet $M^i$ has no subsequence which converges in the Gromov-Hausdorff sense.
\end{cor}

We may also find an example of a sequence that has the same properties as the sequence in Corollary \ref{no GH limit corollary} yet has unbounded area. The following two corollaries, consequences of Theorem \ref{third theorem}, are summarized in Figure \ref{entropy example}.

\begin{cor}\label{no GH limit with area unbounded}
There exists a sequence of closed hypersurfaces $M^i\subset \mathbb{R}^{n+1}$, such that $\diam(M^i)$ is uniformly bounded and each flow $M^i_t$ exists on a uniform time interval $[0,T]$ and shrinks to a round point, yet $\Area(M^i) \to \infty$ and $M^i$ has no subsequence which converges in the Gromov-Hausdorff sense. Moreover, for each $t$, there is $C(t)$ such that $\Area(M^i_t)<C(t)$.
\end{cor}

\begin{figure}
\centering
\includegraphics[scale = .45]{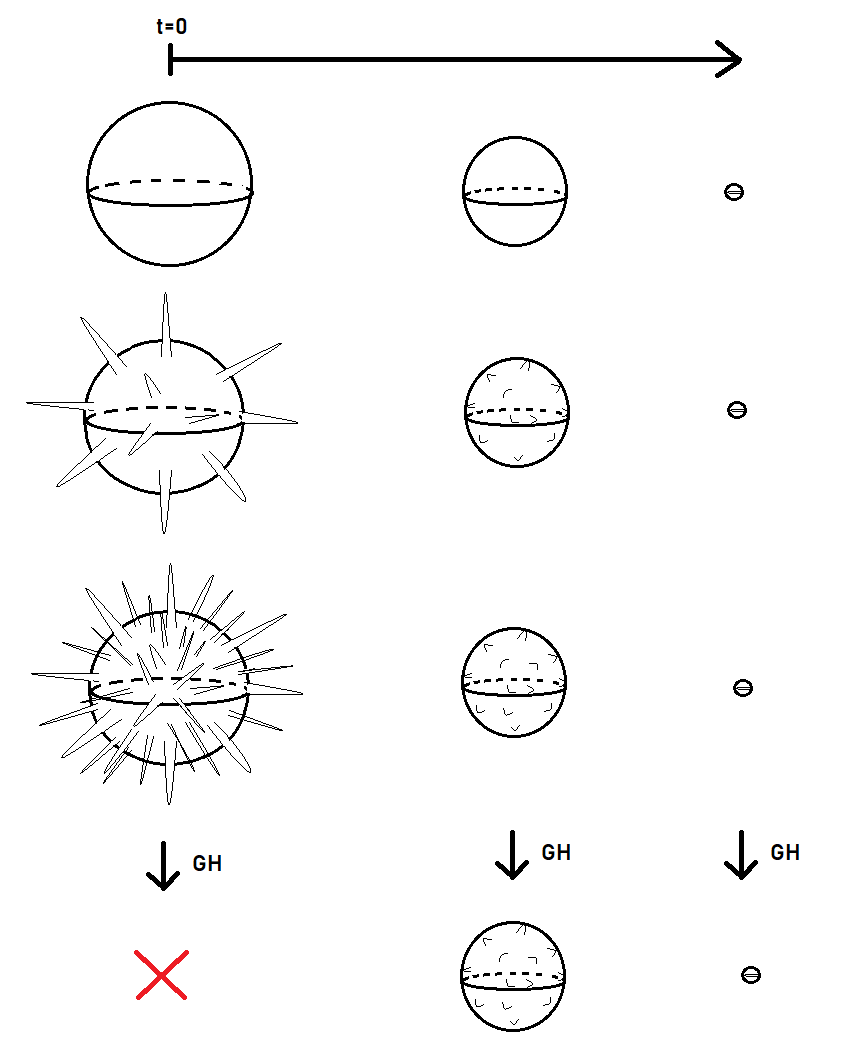}
\caption{A sequence of flows shrinking to a round point yet the sequence of initial surfaces has no Gromov-Hausdorff limit.}
\label{spiky ball figure}
\end{figure}

\begin{figure}
\centering
\includegraphics[scale = .45]{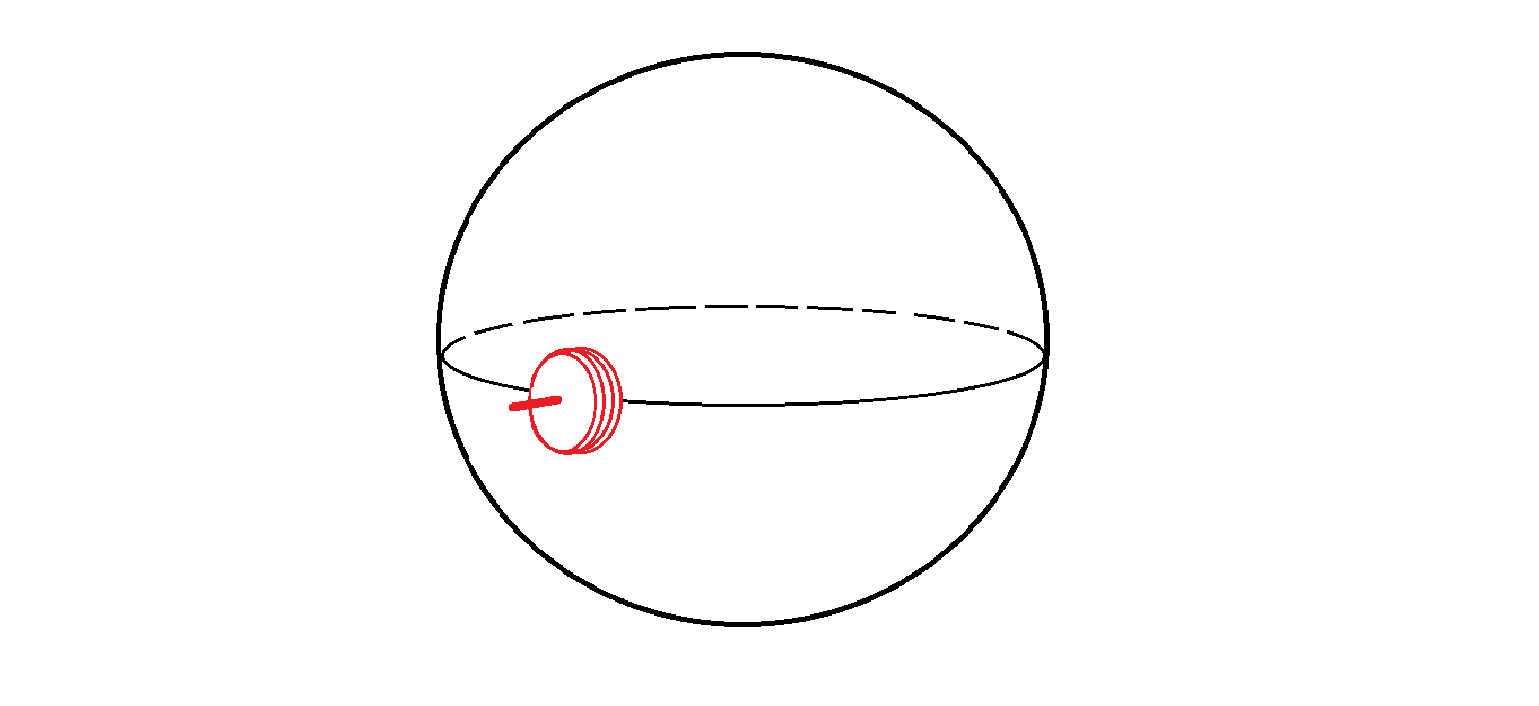}
\caption{A surface that is close in Hausdorff distance to a sphere and shrinks to a round point yet has high entropy and large area.}
\label{entropy example}
\end{figure}

Since we can construct surfaces which will have arbitrarily high area in a compact region, we can find examples of arbitrarily high entropy surfaces, in the sense of Colding-Minicozzi~\cite{ColdingMinicozziGeneric}, which smoothly flow to round points. On the other hand, Bernstein and L. Wang's landmark theorem (\cite{BernsteinWangHausdorffstability18}, see also the generalization by S. Wang in \cite{shengwenwang2020round}) says that closed surfaces in $\mathbb{R}^3$ of low entropy are close in Hausdorff distance to the round sphere. Here we show that the converse to their theorem is wildly false even if one assumes the surface flows smoothly to a round point. That is, we will construct surfaces that are arbitrarily Hausdorff close to the round sphere and flow to round points, yet have arbitrarily large entropy. To do this, we modify the construction in the above corollary to be as close as we want in Hausdorff distance to a round sphere and have arbitrarily large entropy despite flowing to a round point. This is the content of the following corollary, which follows from the construction in Corollary \ref{no GH limit with area unbounded}. As usual, we denote the entropy of $M$ by $\lambda(M)$. 

\begin{cor}\label{highentropy} For every $\delta > 0$ and $E > 0$, there exists a closed hypersurface $M \subset \mathbb{R}^{n+1}$ which shrinks to a round point and is $\delta$-close in Hausdorff distance to the round sphere, yet $\lambda(M) > E$. 
\end{cor}

We may generalize Corollary \ref{no GH limit with area unbounded} and thus generalize a result of Joe Lauer~\cite{lauer2013newlength} as well as a result of the first named author \cite{Mramor}. Lauer showed that there are sequences of closed embedded curves $\gamma_i$ that limit to a space-filling curve, yet applying the curve shortening flow to each $\gamma_i$ for some time $t$ gives a uniform bound on length $\text{Length}(\gamma_i) < C(t)$. We will prove a higher-dimensional version of this for mean curvature flow. Our arguments for Corollary \ref{space filling surface sequence} do not work for curve shortening flow.

\begin{cor}\label{space filling surface sequence}
There exists a sequence of closed hypersurfaces\footnote{$M_i$ is technically an immersed hypersurface. However, we will often work with its image as a subset of $\mathbb{R}^{n+1}$, which we again denote by $M^i$. This is to have more concise notation, without any loss of rigor.} $M^i$ in $\mathbb{R}^{n+1}$, such that each flow $M^i_t$ exists on a uniform time interval $[0,T]$ and shrinks to a round point, yet $M^i$ limits to a space-filling surface containing the unit ball\footnote{By this, we mean that $M^i$ converges in the Hausdorff distance to some set $\mathcal{K}\subset \mathbb{R}^{n+1}$ such that $\mathcal{K}$ contains a unit ball $\mathcal{B}$. In particular, this means that for each $x \in \mathcal{B}$, there is a sequence of $x_i \in M^i$ such that $\lim_i x_i = x$.}. Moreover, for each $t > 0$, there is $C(t)$ such that $\Area(M^i_t)<C(t)$. 
\end{cor}

Before moving on, we point out that the corollaries above may be interpreted as statements regarding the basin of attraction of the round sphere for the mean curvature flow. Thinking of mean curvature flow in a dynamical sense, these corollaries show that the basin of attraction for the round sphere is much more complicated than simply the convex surfaces. In particular, it is not compact under any reasonable topology. 
$\medskip$

To end, we generalize Corollary \ref{space filling surface sequence} to surfaces which do not necessarily shrink to round points. The idea is that for any closed embedded hypersurface $M$, we may find a sequence $M^i$ that limits to a space-filling surface covering the region bounded by $M$, $\text{int}(M)$, and the flows $M^i_t$ approximate the flow $M_t$ for as long as $M_t$ has bounded curvature. As in the construction in the corollary above, one can also arrange these examples to have arbitrarily large area.

\begin{cor}\label{space filling hypersurface}
Let $M\subset \mathbb{R}^{n+1}$, be a closed embedded hypersurface. Suppose that the flow $M_t$ has bounded second fundamental form for time $[0,T]$.

Then, for each $\epsilon>0$, there exists a sequence of closed hypersurfaces\footnote{Here, $M_i$ is technically an immersed hypersurface. See the footnote related to Corollary \ref{space filling surface sequence}.} $M^i$ in $\mathbb{R}^{n+1}$ such that $M^i$ limits to a space-filling surface containing $\text{int}(M)$, each flow $M^i_t$ exists on a uniform time interval $[0,T^*]$, and for some $t_0\in (0,\text{min}(T,T^*))$, $M^i_t$ is $\epsilon$-close to $M_t$ in $C^2$ for all $t\in [t_0, \text{min}(T,T^*)]$. We may find $t_0$ such that $t_0 \to 0$ as $\epsilon \to 0$.
\end{cor}

\begin{figure}
\centering
\includegraphics[scale = .4]{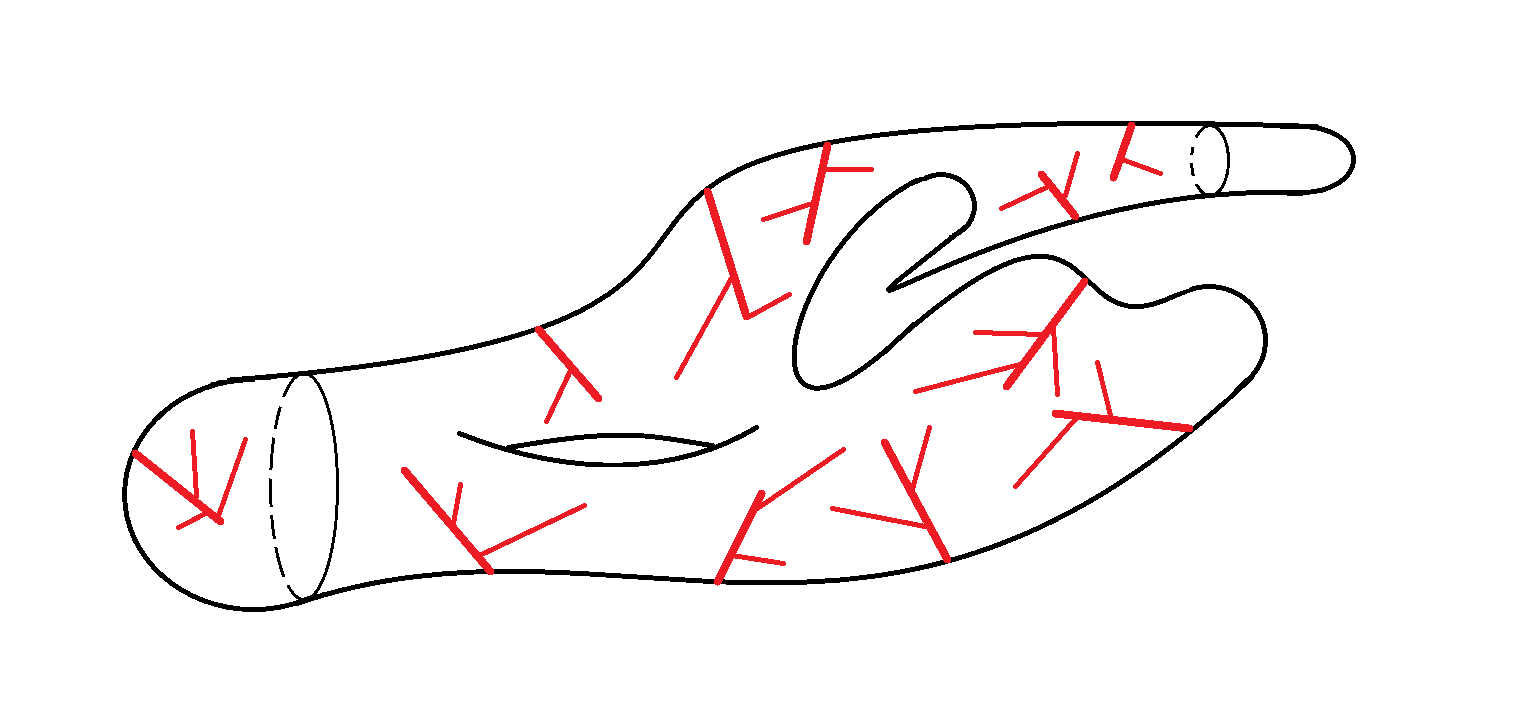}
\caption{A sketch of one of the first elements in the sequence described in Corollary \ref{space filling hypersurface}. The surfaces in our construction are immersed and many of the spikes will intersect, although the construction can likely be altered to preserve embeddedness.}\label{space filling arbitrary surface}
\label{figure 4}
\end{figure}

Figure \ref{space filling arbitrary surface} roughly encapsulates how the sequences in both Corollary \ref{space filling surface sequence} and Corollary \ref{space filling hypersurface} are constructed; we construct the sequence by iteratively adding inward-pointing spikes at smaller and smaller scales.
$\medskip$

Lastly, in the appendix,  using an argument, we prove the mean curvature flow analogue of a result due to Petersen-Tao for Ricci flow~\cite{PetersenTao09} . 
\begin{customthm}{A.2}\label{appendix theorem in intro} Let $\Sigma(d,C)$ be the set of closed embedded hypersurfaces $M^n \subset \mathbf{R}^{n+1}$ such that 
\begin{enumerate}
\item $\diam(M) <d$
 \item $|A|^2< C$
 \end{enumerate}
 Then there exists an $\epsilon(d,C) > 0$ such that if $M \in \Sigma(d,C)$ and $k_{min} > -\epsilon(d,C)$, then $M_t$ flows smoothly to a round point.
  \end{customthm}

This is more general than Theorem \ref{first theorem} since it concerns all almost convex surfaces, as opposed to just neighborhoods of line segments. However, Theorem \ref{appendix theorem} is proven via compactness-contradiction argument and as such it does not seem clear how to effectively estimate the constant $\epsilon$. Theorem \ref{first theorem} is better in this regard because it gives us a more precise understanding of the constants that arise. Note that the class of examples produced via Theorem \ref{second theorem} are likely unattainable via any compactness-contradiction argument on its own as there is no uniform curvature bound on members of that set. 

$\medskip$

$\textbf{Acknowledgements:}$ We are grateful to Bruce Kleiner, the second author's advisor, for encouraging us to construct more pathological examples than we initially had and in particular for suggesting a version of the construction involved in Theorem \ref{third theorem}. The first author additionally thanks his advisor, Richard Schoen, for his support and valuable advice. We are also grateful to the referees for their comments and suggestions which helped improve the clarity of the article.

\section{Preliminaries, Old and New, on the Mean Curvature Flow.}
In this section we collect some standard and perhaps slightly less standard facts and observations on the mean curvature flow which we will employ in the subsequent sections. Let $M$ be an $n$-dimensional 2-sided manifold and let $F: M \to \mathbf{R}^{n+1}$ be an embedding of $M$ realizing it as a smooth closed hypersurface of Euclidean space, which by abuse of notation we also refer to as $M$. Then the mean curvature flow of $M_t$ is given by the image of $\hat{F}: M \times [0,T) \to \mathbf{R}^{n+1}$ satisfying
\begin{equation}\label{MCF equation}
\frac{d\hat{F}}{dt} = H \nu, \text{ } \hat{F}(M, 0) = F(M)
\end{equation}
where $\nu$ is the inward pointing normal and $H$ is the mean curvature. It turns out that (\ref{MCF equation}) is a nonlinear heat-type equation, since for $g$ the induced metric on $M$,
\begin{equation}
 \Delta_g F = g^{ij}(\frac{\partial^2 F}{\partial x^i \partial x^j} - \Gamma_{ij}^k \frac{\partial F}{\partial x^k}) = g^{ij} h_{ij} \nu = H\nu
\end{equation}
One can easily see that the mean curvature flow equation (\ref{MCF equation}) is degenerate. Despite this, solutions to (\ref{MCF equation}) always exist for short time and are unique. There are several ways to deduce this by relating (\ref{MCF equation}) to a nondegenerate parabolic PDE. Solutions to the mean curvature flow satisfy many properties that solutions to heat equations do, such as the maximum principle and smoothing estimates. One important consequence of the maximum principle is the comparison principle (also known as the avoidance principle), which says that two initially disjoint hypersurfaces will remain disjoint over the flow.
$\medskip$ 

Generally speaking, the mean curvature flow cannot be written down explicitly except in cases with a high amount of symmetry. For example, the round sphere shrinks by dilations to a point in finite time. By the avoidance principle, we may use the sphere as a barrier and see that any compact surface develops a singularity in finite time over the flow. It is interesting of course to understand when the manifold shrinks to a point, i.e.\ when there are no ``leftover'' regions of low curvature at the singular time. This generally does not happen---neckpinches occur quite generally---but in some cases it does. The classical result of Huisken gives a simple condition for this to happen~\cite{Huisken84}.
\begin{thm}[Huisken~\cite{Huisken84}] Convex closed hypersurfaces remain convex under mean curvature flow and shrink to a round point in finite time.
\end{thm} 

The goal of this paper is to extend Huisken's theorem to certain special types of non-convex surfaces. First, we state the one-sided minimization property of mean convex flows, discovered by Brian White~\cite{Wh03}.

\begin{thm}[White~\cite{Wh03}]\label{theorem one sided min} Let $K_t$ be a flow of domains where $\partial K_t$ are mean convex and evolve by mean curvature flow. Let $B = B(x,r)$ be a ball, and let $S$ be a slab in $B$ of thickness $2\epsilon r$ passing through the center of the ball, i.e.
\begin{equation}
 S = \{ y \in B \mid dist(y, H) < \epsilon r \}
 \end{equation} 
where $H$ is a hyperplane passing through the center of the ball and $\epsilon > 0$. 
Suppose $S$ is initially contained in $K$, and that $\partial K_t \cap B$ is contained in the slab $S$. Then $K_t \cap B \setminus S$ consists of $k$ of the two connected components of $B \setminus S$, where $k$ is $0, 1$, or $2$. Furthermore,
 \begin{equation} 
 \Area(\partial K_t \cap B) \leq (2 - k + 2n \epsilon) \omega_n r^n. 
 \end{equation} 
 \end{thm} 

Our extensions of Huisken's theorem will use some barrier arguments. The next proposition is a basic observation about barriers, but it is centrally used in the proof of Theorems \ref{first theorem}. It will be used in conjunction with Proposition \ref{smooth} below. 
\begin{prop}\label{subsolution proposition} Let $M$ be a closed hypersurface in $\mathbb{R}^{n+1}$ corresponding to the boundary of a domain $K$. Let $\tilde{M}$ be a closed hypersurface in $\mathbb{R}^{n+1}$, disjoint from $M$, corresponding to the boundary of a domain $\tilde{K} \supset K$. Denote the mean curvature flow of $M$ by $M_t$, and denote by $\widetilde{M}_t$  the flow of $M$ with speed function $X(p,t)$ satisfying
\begin{equation}
 X(p,t) \leq H(p,t)
 \end{equation}
Denote the flows of the corresponding domains by $K_t$ and $\widetilde{K}_t$, respectively. Then, for any $T$ such that $M_t$ and $\tilde{M}_t$ are smoothly defined for $[0,T]$, $K_T \subseteq \widetilde{K}_T$.
$\medskip$

Analogously, if $H \leq X$ and $\tilde{K}\subset K$, then $\widetilde{K}_T \subseteq K_T$. 
 \end{prop} 
 \begin{proof}
Suppose that for some time $t_0> 0$, $M_{t_0} \cap \widetilde{M}_{t_0} \neq \emptyset$. Then, since $M_t$ is a mean curvature flow and $X \leq H$ on $\widetilde{M}_t$, there exists $\epsilon>0$ such that $\widetilde{M}_t \cap M_t = \emptyset$ for $t_0 < t < t_0 + \epsilon$. This follows from a standard argument using the strong maximum principle (see \cite[Thm. 2.2.1]{Mantegazza}). This shows that $K_t \subseteq \widetilde{K}_t$ so long as each flow is smooth. The other direction follows analogously.
 \end{proof}
 
With Proposition \ref{subsolution proposition} in mind, throughout this paper we will use the following definition:
\begin{defn}
If a flow of surfaces satisfies $X \leq H$ as in the above theorem, then it is called a subsolution. On the other hand, if $X\geq H$, then it is called a supersolution.
\end{defn}

In the sequel we will need some facts particular to mean convex flows which we discuss next. 
$\medskip$

 Our first goal is to prove Proposition \ref{smooth}, which will use the Brakke regularity theorem, originally shown by Brakke in his thesis \cite{Brakke1978}. The following version of the regularity theorem, simpler to state, is due to Brian White \cite{White05}. It is true for smooth flows up to their first singular time but can be used to rule out singularities in a short forward period of time (and hence iterated) if it is applicable at every point of a fixed timeslice of a flow. In the following theorem, we use the notation that $\tau$ is the time function, i.e.\ for a spacetime point $X = (x,t)$, $\tau(X) = t$.

In the following theorem and lemma, we recall that the Gaussian density ratio $\Theta(M_t, X, r)$ is given by
 \begin{equation}
 \Theta(M_t, X, r) = \int_{y \in M_{t - r^2}} \frac{1}{(4\pi r^2)^{n/2}} e^\frac{-|y - X|^2}{4r^2} d \mathcal{H}^n(y)
 \end{equation} 
By Huisken's monotonicity formula~\cite{Huisken90}, this quantity is monotone nonincreasing as $r$ decreases.

 \begin{thm}[White~\cite{White05}]\label{Brakke} There exist $\epsilon_0 = \epsilon_0(n) > 0$ and $C = C(n) < \infty$ with the following property: If $\mathcal{M}$ is a smooth proper mean curvature flow starting from a hypersurface $M$ in an open subset $U$ of the spacetime $\mathbb{R}^{n+1} \times \mathbb{R}$ and if the Gaussian density ratios $\Theta(M_t, X, r)$ are bounded above by $1 + \epsilon_0$ for $0 < r< \rho(X,U)$, then each spacetime point $X = (x,t)$ of $\mathcal{M}$ is smooth and satisfies:
 \begin{equation}
 |A| \leq \frac{C}{\rho(X, U)}
 \end{equation}
 where $\rho(X,U)$ is the infimum of the parabolic distance $||X - Y||$ among all spacetime points $Y \in U^c$, where $\tau(Y) \leq \tau(X)$.
 \end{thm} 
 We will also need the following slight refinement of the statement above.
\begin{lem}\label{refinement} Under the same hypotheses of the theorem above, for every $C' > 0$ there exists $\epsilon_0 = \epsilon_0(C',n)$ so that if $\Theta(M_t, X, r)$ are bounded from above by $1 + \epsilon_0$ for $0 < r< \rho(X,U)$, then $|A| \leq \frac{C'}{\rho(X, U)}$.
 \end{lem}
\begin{proof} To see this we may proceed by contradiction, using estimates from the regular statement of the Brakke regularity theorem above but using essentially the same argument as White. Suppose that the result does not hold. Then, there exists $C'>0$ and a sequence of flows $M^i$ and spacetime points $X_i$ which violate the statement. After recentering $X_i$ to the origin and rescaling by $\rho(X_i, U)$, we may find a sequence of smooth flows $M^i$ in some open set $U$ of spacetime with $|A|^2(0) \geq C'$ but $\Theta(M_t^i, 0, r)$ are bounded above by $1 + \epsilon_i$ for $0 < r< \rho(0,U)$ where $\rho(0,U) > 1$ and $\epsilon_i \to 0$. Note by Brakke regularity and Shi's estimates, we may then pass to a subsequence of flows which smoothly converge to a limit $N$, so that the limit $N$ has nonzero curvature at the origin but has $\Theta(N, 0,r)=1$ for $0<r<\rho(0,U)$. 
$\medskip$
 
By Huisken's monotonicity formula \cite{Huisken90} (see Ecker's local version of Huisken's monotonicity formula~\cite{Ecker01}), we see the limiting flow must satisfy the self shrinker equation in a neighborhood of $0$. Since a priori the curvature is bounded (again, since the regular statement of the Brakke regularity theorem holds) we see then by following the proof of an observation of White (see \cite[Lemma 3.2.17]{Mantegazza}) that $N_t$ is flat near $0$. Indeed, this follows from the fact that $N_t$ satisfies the shrinker equation near $0$, i.e. $\vec{H}(y) = \frac{\langle y, \nu(y)\rangle}{2t}$ for $t <0$ and $y$ near $0$. Multiplying by $2t$ and letting $t \to 0$ and using the uniform bound on curvature, we find that $\langle y, \nu(y) \rangle =0$ for all $y$ near $0$. This implies that $N_t$ is a flat hyperplane in a neighborhood of the origin, which contradicts the fact that $N$ has nonzero curvature at the origin. 
\end{proof}

The idea in the following proposition is to rule out ``microscopic singularities'' similar to the core ideas of \cite{Mramor} and \cite{MramorPayne}. Throughout this paper, we will use the following proposition to squeeze the flow between two better-understood barriers, in order to control the curvature of the flow. This will rule out singularities occuring in between sufficiently close barriers with small curvature, assuming the flow is initially smooth with small graphical norm over the barriers. In the statement, we refer to the flows of domains so as to more clearly distinguish ``inside'' from ``outside.''

 \begin{prop}\label{smooth} Let $B = B(x,r)$ be a ball, and let $K_t$ be a mean curvature flow of domains for $t \in [0,T)$, $T>1$. For $[0,T]$, let $I_t$ and $O_t$ be a flow of domains in $B$, not necessarily via the mean curvature flow, such that $I_t \subset K_t \subset O_t$ in $B$ and $\partial I_t$ and $\partial O_t$ are smooth hypersurfaces with $\partial I_t \cap \partial O_t = \emptyset$. Let $S_t= O_t \setminus I_t$. Suppose that
 \begin{enumerate}
 \item $K_t \cap B \setminus S_t$ is exactly\footnote{We must also stipulate that $B$ is large enough so that these two sets are nonempty.} $I_t \cap B \setminus S_t$
 \item there exists $C > 0$ so that $|A|^2 < C$ on $I_t$, $O_t$, and $\partial K_t \setminus B(x, \frac{r}{2})$ for $t \in [0,T]$.
\end{enumerate}
\noindent Then, for every $\rho >0$, there exists $\delta^* > 0$ with the following properties: Suppose that for $0 < \delta < \delta^*$,

\begin{enumerate}[resume]
\item $\partial I_t$ can be written as a graph $g$ over $\partial O_t$ with $|g|_{C^2(\partial O_t)} < \delta$ for $t \in [0,T]$, and
\item $\partial K_0$ can be written as a graph $f$ over $\partial I_0$ and $\partial O_0$ with $|f|_{C^2(\partial I_0)} < \rho$.
\end{enumerate}
Then,  $\partial K_t$ exists smoothly for $[0,T]$ and for $t \in [0,T]$, there exist graphs $f^1_t, f^2_t$ defined over $\partial I_t \cap B(x, \frac{r}{2})$ and $\partial O_t \cap B(x, \frac{r}{2})$, respectively, such that $\partial K_t \cap B(x, \frac{r}{2})$ coincides with $f^1_t, f^2_t$ and $||f^1_t||_{C^2(\partial I_t)}, ||f^2_t||_{C^2(\partial O_t)}\leq 2\rho$.
 \end{prop}

 \begin{proof} 
There exists a graph $f^1_t$ over $\partial I_t$ and $s\ll 1$ depending on $C$ such that the graph of $f_t$ coincides with $\partial K_t \cap B(x, r)$ and $||f^1_t||_{C^2(\partial I_t)} < 2\rho$ for $t \in [0,s]$. We may take $s$ to be the doubling time for $\partial K_0$, so it only depends on $C$ and $\rho$. Note that for this choice of $s$, we may choose $\delta^*$ small enough so that $\partial K_t \cap B$ will remain a graph over $\partial I_t$ for $t \in [0,s]$, and so there exists the desired $f^1_t$. This follows since $\partial K_t\cap B$ remains within the $\delta^*$-neighborhood of $\partial I_t$, which combined with the curvature bound on $\partial I_t$ and $\partial K_t$ over $t \in [0,s]$ (since $s$ is the doubling time) implies that $\partial K_t \cap B$ remains graphical over $\partial I_t$ for $t \in [0,s]$. 
$\medskip$

We now apply the Brakke regularity theorem to $\partial K_t \cap B$ for $t \in [0,s]$ to improve the $C^2$ estimate on $f^1_t$. From (1) and (3), we have that $||f^1_t||_{C^0(\partial I_t)} < \delta$ in $B$ for $t \in [0,T]$. Combined with the fact that $||f^1_t||_{C^2(\partial I_t)} < 2\rho$ in $B$ for $t \in [0,s]$, we have that $||f^1_t||_{C^1(\partial I_t)} < M\rho \delta^*$ in $B$ for $t \in [0,s]$, where $M$ is a constant depending on $C, \rho$. By choosing $\delta^*$ small enough, this implies that for the ball $B$, the area of $\partial K_t$ as a graph over $B \cap \partial I_t$ is arbitrarily close to the area of $B \cap \partial I_t$. That is, the $C^1$ estimate on $f^1_t$ goes to zero as $\delta^* \to 0$. So, for $0< \gamma < 1$ (to be chosen later) and each $r^*\ll 1$, we find that $\Theta(\partial K_t \cap B, x, r^*)$ is arbitrarily close to $\Theta(\partial I_t \cap B, x, r^*)$ for $x \in B(x, \gamma r)$ and $\delta^*$ chosen small enough. We then apply the Brakke regularity theorem, particularly Lemma \ref{refinement} in in $B(x, \gamma r)$, over uniformly small scales $r^*$ to find that $||f^1_s||_{C^2(\partial I_s)} < \frac{3\rho}{2}$ in $B(x, \gamma r)$. Then, let $s^*$ be the time such that there exists a graph $f^1_t$ over $\partial I_t$ which coincides with $\partial K_t \cap B(x,\gamma r)$ and $||f^1_t||_{C^2(\partial I_t)} < 2\rho$ in $B(x, \gamma r)$ for $t \in [0,s+s^*]$. By choosing $\gamma > \frac{1}{2}$, $s^*$ depends only on $C, \rho$, since $\partial K_t \setminus B(x, \frac{r}{2})$ satisfies $|A|^2 < C$. Applying the same reasoning as above, we get that $||f^1_t||_{C^1(\partial I_t)} < M'\rho \delta$ in $B(x, \gamma r)$ for $t \in [0, s+s^*]$, where $M'$ depends on $C, \rho$. Applying Lemma \ref{refinement} again with $\delta$ small enough, we find that $||f^1_t||_{C^2(\partial I_t)} < \frac{3\rho}{2}$ in $B(x, \gamma r)$ for $t \in [s+s^*]$. We then iterate this argument for all $\gamma$ close enough to $1$ to find that $||f^1_{t}||_{C^2(\partial I_t)}<2 \rho$ in $B(x, \gamma^N r)$ for $t \in [0,T)$, where $N$ is the number of iterations. Note that in each iteration, $s^*$ depends only on $C, \rho$, as it is a doubling time, so this argument need only be iterated finitely many times. Then, if we choose $\gamma$ such that $\gamma > \frac{1}{2^{1/n}}$, this concludes the result for $f^1_t$. 
\medskip

An identical argument works to prove the that the same result holds for a graph $f^2_t$ over $\partial O$. This implies that $\partial K_t$ is smooth with uniformly bounded curvature everywhere for $t \in [0,T)$ and so in particular $\partial K_t$ exists smoothly for $t \in [0,T]$. 

\end{proof}

We conclude this section with a discussion about pseudolocality in mean curvature flow. Pseudolocality says that the mean curvature flow at some point is controlled for short time by the curvature in a ball around that point, and far away parts of the flow affect the flow around the point very little. Pseudolocality plays a crucial role in our  arguments in the next couple of sections, such as in the proof of Lemma \ref{nearly graphical lemma}. The following theorem due to Chen and Yin (\cite[Theorem 7.5]{Chen2007}), which is adapted to the particular case of ambient Euclidean space, underpins our usage of pseudolocality in this paper. See also the more general pseudolocality result of Ilmanen-Neves-Schulze~\cite{INS}.

\begin{thm}[Chen-Yin~\cite{Chen2007}]\label{pseudolocality} There is an $\epsilon_*>0$ with the following property. Suppose we have a smooth solution $M_t \subset \mathbb{R}^{n+1}$ to mean curvature flow properly embedded in $B(x_0, r_0)$ for $t\in [0,T]$ where $T\leq \epsilon_*^2 r_0^2$. We assume that at time zero, $x_0 \in M_0$, the second fundamental form satisfies $|A|(x) \leq r_0^{-1}$ on $M_0 \cap B(x_0, r_0)$, and $M_0$ is graphical in the ball $B(x_0, r_0)$. Then,
\begin{equation}|A|(x,t)\leq (\epsilon_* r_0)^{-1}\end{equation}
for any $x \in B(x_0, \epsilon_* r_0) \cap M_t$ for $t\in [0,T]$.   

\end{thm}

 If there are additional initial bounds for $|\nabla A|$ and $|\nabla^2 A|$ then we obtain bounds on $|\nabla A|$ and $|\nabla^2 A|$ for short time using Theorem \ref{pseudolocality} in combination with an application of \cite[Lemma 4.1, 4.2]{BrendleHuisken16}. In Section 4, Theorem \ref{pseudolocality} will be used in combination with the evolution equations of $H$ and the shape operator, which involve diffusion terms which are second order in the curvature. As in \cite{Mramor}, we will use pseudolocality to ensure that the flow of two-convex domains in a ball remains two-convex for some time in a slightly smaller ball so long as some curvature control is assumed near the boundary. This will be explained further in Section 4.

\section{Proof of Theorem 1.1}
In this section we show Theorem \ref{first theorem}, i.e.\ that there are neighborhoods of some embedded, nonconvex intervals that smoothly shrink to round points under the mean curvature flow.
$\medskip$
  
We will construct supersolutions and subsolutions to the flow, which will be barriers for the flow by Proposition \ref{subsolution proposition}. The point is that we will construct these barriers explicitly, and this will give us enough control over the true flow in order to apply Proposition \ref{smooth} and obtain the statement.
\medskip

Fix a length $L$ throughout this argument, and let $\Sigma(n,L)$ be the set of smooth embedded intervals with length bounded by $L$. Recall that $T_r(\gamma)$ is the boundary of the radius-$r$ neighborhood of $\gamma$.
$\medskip$

Let $\gamma \in \Sigma(n,L)$, such that $|A|, |\nabla A|\leq C$ on $\gamma$, where $C$ is to be chosen sufficiently small later.

Let $\gamma$ be arclength parametrized over an interval of length $L$, which we fix to be $I \subset \mathbb{R}^{n+1}$, defined as
$$I:= \{(s,0,\dots, 0)\,|\, s \in [-L/2,L/2]\}$$

Let $\Lambda:=T_r(I)$ be the boundary of the radius $r$-neighborhood of the interval $I$, for some $r$ to be chosen small later. Similarly, let $\Gamma := T_r(\gamma)$.  
$\medskip$

The idea is that for $r$ small enough (relative to the curvature of $\gamma$), the flow of a smooth surface $C^1$-close to $\Gamma$ should be closely approximated by the flow of a smooth convex surface $C^1$-close to the convex tube $\Lambda$, as seen in Figure \ref{Figure tube}, after a standard map from $\Lambda$ to $\Gamma$. This mapping is given by extending to tubular neighborhoods the map from $I$ to the curve $\gamma$. Since the convex surface close to $\Lambda$ will shrink to a point, irrespective of its length $L$, it is reasonable to expect that the flow of a surface near $\Gamma$ will too. However, to make this argument work, we must choose the curvature bound $C$ on $\gamma$ sufficiently small, as well as a sufficiently small $r$ and close enough surfaces to $\Lambda$ and $\Gamma$. 
$\medskip$

\begin{figure}
\centering
\includegraphics[scale = .5]{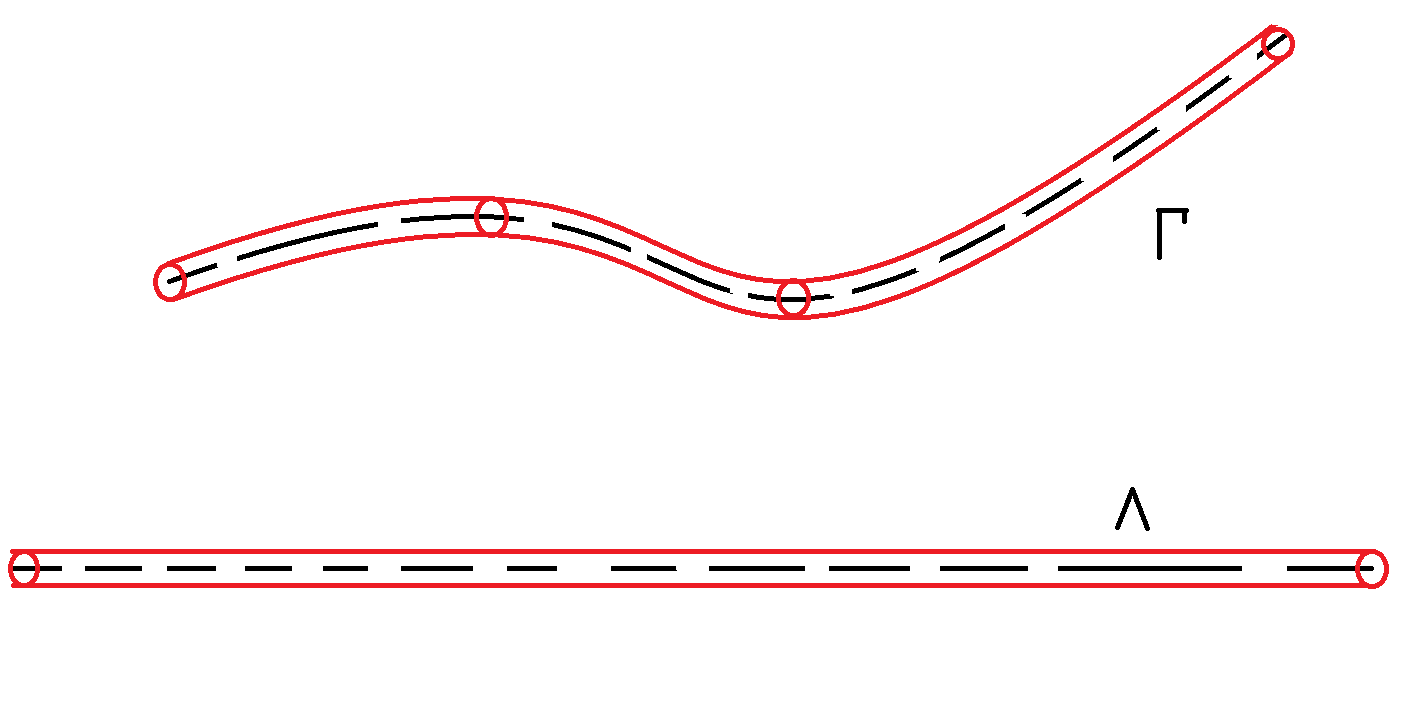}
\caption{}
\label{Figure tube}
\end{figure}

 Let $0<r<r^*$, to be chosen sufficiently small later. Define
$$\mathcal{I}:= \{(s,0, \dots, 0) \,|\,s \in [-r^* -L/2, r^* +L/2]\}$$
By abuse of notation, we will occasionally conflate $\mathcal{I}$ with $[-r^* -L/2, r^* +L/2 ]$.

First, we define $\tilde{\gamma}$ as an extension of $\gamma$ by straight line segments. That is, we define 
\[
    \tilde{\gamma}(s):= 
\begin{cases}
   \gamma(s),& \text{if } s \in [-L/2,L/2]\\
    \gamma(-L/2) +(s+L/2) \gamma'(-L/2),              & \text{if } s \in [-r^* -L/2,-L/2]\\
    \gamma(L/2) +(s-L/2)\gamma'(L/2),                 & \text{if } s \in [L/2,r^* +L/2]
\end{cases}
\]

For $s \in \mathcal{I}$, find a continuously varying set of orthonormal basis vectors 
$$\{\tilde{e}_2(s), \dots, \tilde{e}_{n+1}(s)\}$$ 
in the normal bundle of $\tilde{\gamma}$, and define for $t \leq r^*$, 
$$\tilde{\psi}(s,t\mathbf{a}) := \tilde{\gamma}(s)+t\sum_{i=2}^n a_i \tilde{e}_i(s)$$
where $s  \in\mathcal{I}$ and $\mathbf{a} = (a_2, \dots, a_n) \in \mathbb{R}^{n}$ with $|\mathbf{a}|=1$. This defines $\tilde{\psi}$ on the $r^*$-tubular neighborhood of $\mathcal{I}$ and its image is the $r^*$-tubular neighborhood of the curve $\tilde{\gamma}$. We may choose $C$ and $r^*$ small enough, depending on $L$, so that $\tilde{\gamma}$ is a $C^1$-diffeomorphism from the $r^*$-tubular neighborhood of $\mathcal{I}$ to the $r^*$-tubular neighborhood of $\tilde{\gamma}$. Recall that $\Lambda := T_r(I)$. With a small enough choice of $C, r^*$, $\tilde{\psi}(\Lambda)$ is an embedded $C^1$ hypersurface. 
\medskip

By construction, 
\begin{equation}\label{equation gamma tube}
\tilde{\psi}(\Lambda) = T_r(\gamma)
\end{equation}

Now, let $\bar{\gamma}$ denote a smooth arclength parametrized curve defined on $I$, such that $|A|, |\nabla A| \leq C$, $\bar{\gamma}(s) = \gamma(s)$ for $s \in [0,L]$, and $\bar{\gamma}$ is $\eta$-close in $C^1$ to $\tilde{\gamma}$, for $\eta$ to be chosen later.
\medskip

For $s \in \mathcal{I}$, find a smoothly varying set of orthonormal basis vectors $\bar{e}_2(s), \dots, \bar{e}_{n+1}(s)$ in the normal bundle of $\bar{\gamma}$. Then, for $t\leq r^*$, define
 $$\bar{\psi}(s,t\mathbf{a}) := \bar{\gamma}(s)+t\sum_{i=2}^n a_i \bar{e}_i(s)$$
where $s  \in \mathcal{I}$ and $\mathbf{a} = (a_2, \dots, a_n) \in \mathbb{R}^{n}$ with $|\mathbf{a}|=1$.

By (\ref{equation gamma tube}), we find that $\bar{\psi}(\Lambda)$ approaches $T_r(\gamma) = \tilde{\psi}(\Lambda)$ in $C^1$-norm as $\eta \to 0$. For sufficiently small $C, r^*$ as well as small enough $\eta$, we find that $\bar{\psi}(\Lambda)$ can be found arbitrarily close to $T_r(\gamma)$ in $C^1$. Note that the choices of $\tilde{e}_i, \bar{e}_i$ are irrelevant as $\Lambda$ is rotationally symmetric with respect to the $x_1$-axis coinciding with the axis of $I$.

 \begin{lem}\label{lemma embedded tube} For each $L$, there exists $C, r^*$ with the following significance: For each $\gamma \in \Sigma(n,L)$ with $|A| \leq C$ and each $0<r<r^*$ and $\epsilon \ll 1$, there exists a smooth embedded strictly mean convex hypersurface $\bar{M}$ $\epsilon$-close in $C^1$ to $\Gamma = T_r(\gamma)$.

 \end{lem}
 \begin{proof}

For fixed $L$, we may choose $C\leq \frac{\pi}{L}$. This means the maximal radius of curvature of $\gamma$ is $\frac{L}{\pi}$. By integrating the curvature bound $|A|\leq C$ and using the fact that the radius of curvature is bounded by $\frac{L}{\pi}$, we have that for $0<r<\frac{L}{2\pi}$, $T_r(\gamma)$ is embedded. 
\medskip

For $0<r<r^*$, let $M(r)$ be a smooth convex hypersurface which is rotationally symmetric around the $x_1$-axis (i.e., the axis coinciding with $I$) and which is $\eta$-close in $C^{1,1}$ to $\Lambda:= T_r(I)$. Define $\Lambda^- \subset \Lambda$ to be the subset consisting of $x \in \Lambda$ such that $x \in N\gamma$, the normal bundle of $\gamma$. That is, $\Lambda^-$ is the smooth subset of $\Lambda$ not including the convex caps.  We may choose $M(r)$ so that $M(r)$ is $\eta$ close in $C^{\infty}$ to $\Lambda^-$. This is possibe since $\Lambda$ is $C^{1,1}$ and $\Lambda^-$ is smooth. Now, let 
$$\bar{M}:= \bar{\psi}(M(r))$$

We may choose $C$ and $r^*$ small enough, with $\eta \ll C,r^*$, so that $\bar{M}$ is mean convex. Notice that $\bar{M}$ is convex near the boundary of the interval $\gamma$. Let $x \in \bar{\psi}(\Lambda^-)$, and let the smallest principal curvature of $\bar{M}$ be $\kappa_1$. Choosing $r^*$ small, we get that $\kappa_1(x)\geq -2C$. Also, choosing $C, \eta$ small enough, the second-smallest principal curvature satisfies $\kappa_2(x)\geq \frac{1}{2r^*}$. This implies that $\bar{M}$ is mean convex.
$\medskip$ 

Now, choosing $\eta$ small enough, we have that $\bar{M}$ is $\epsilon/2$-close in $C^1$ to $\bar{\psi}(\Lambda)$. Since $\bar{\psi}(\Lambda)$ is $\eta$-close in $C^1$ to $\Gamma = T_r(\gamma) = \tilde{\psi}(\Lambda)$, by choosing $\eta$ small enough, we have that $\bar{M}$ is $\epsilon$-close in $C^1$ to $\Gamma$. Indeed, since $T_r(\gamma)$ is embedded, $\bar{M}$ is embedded as well.

 \begin{figure}
 \centering
\includegraphics[scale = .33]{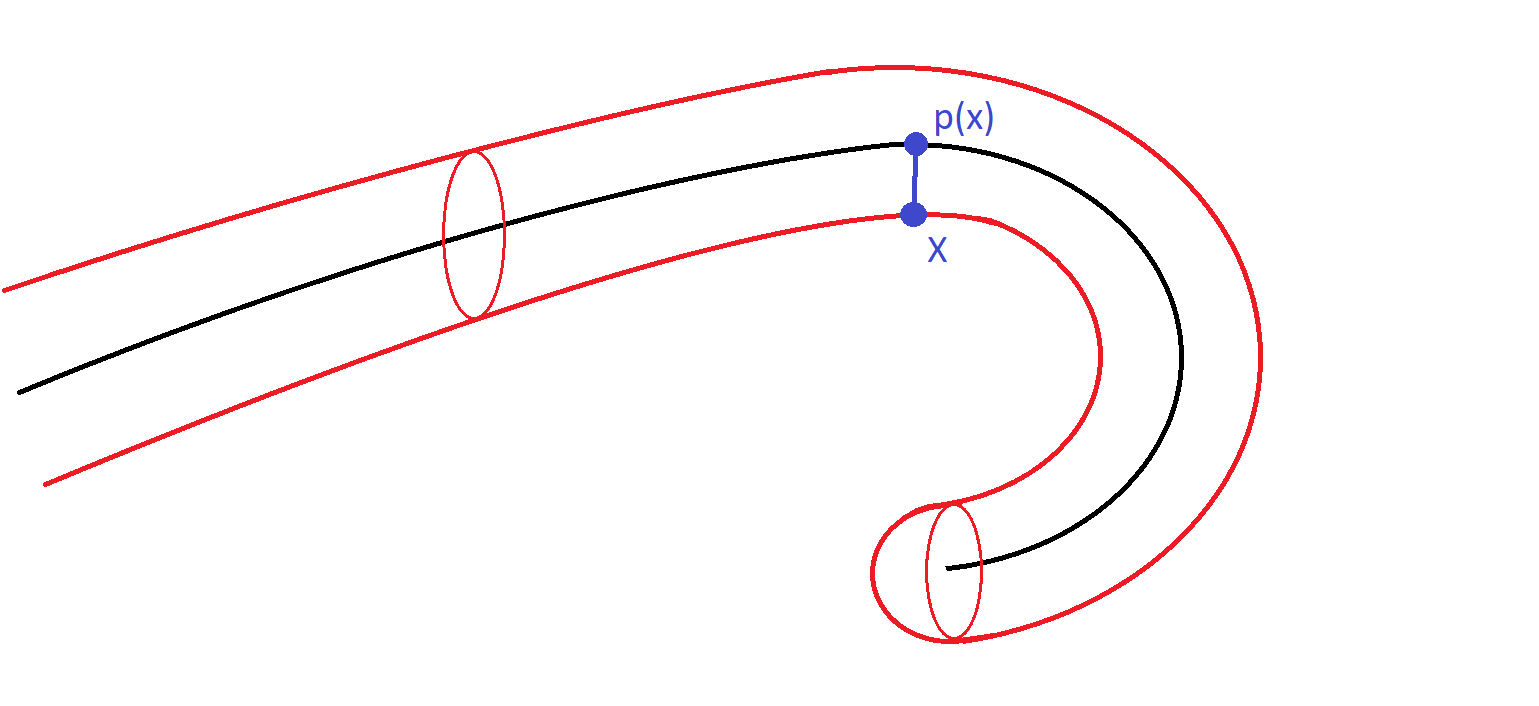}
\caption{A depiction of $T_r(\gamma)$. The point $x \in T_r(\gamma)$ and $p(x) \in \gamma$ such that $x-p(x)$ belongs to the normal bundle of $\gamma$ and $|x - p(x)|=r$.}
\label{tube diagram}
\end{figure}

\end{proof}

Next, we describe flows of inner and outer barriers for smooth perturbations of $\Lambda$ (that is, $M(r)$ from Lemma \ref{lemma embedded tube}), which we denote by $A_t$ and $B_t$, respectively. These will then be mapped, via $\bar{\psi}$, to our intended inner and outer barriers for the flow of $\bar{M}$. These objects furthermore depend on a choice of $\epsilon \ll 1$, but this can be chosen afterwards much smaller than any other chosen parameter so that its effect is negligible for any of the ensuing estimates.
$\medskip$

We will evolve the barrier $A_t$, with initial condition $M(r)$, by a slightly sped up mean curvature flow of $M(r)$. That is, for $\delta>0$, to be chosen later, let $A_t$ evolve by the flow:
\begin{equation}
\frac{dF_\delta}{dt} = e^{\delta}H \nu
\end{equation}
with initial condition $A_0 = M(r)$. Similarly, for $\delta>0$, let $B_t$ evolve by the flow:
\begin{equation}
\frac{dF_\delta}{dt} = e^{-\delta}H \nu
\end{equation}
with initial condition $B_0 = M(r)$. 
These two flows are just mean curvature flow on slightly faster or slower time scales, respectively. We will show that with appropriate choices of $C$ and $r^*$, $\bar{\psi}(A_t)$ and $\bar{\psi}(B_t)$ are a supersolution and subsolution to the flow, respectively.

$\medskip$

Now, we will show that for each $\delta>0$, there is an appropriate choice of small enough $C,r$ such that $\bar{A}_t:=\bar{\psi}(A_t)$ and $\bar{B}_t:= \bar{\psi}(B_t)$  will be appropriate inner and outer barriers for the flow $\bar{M}_t$. Here, $\bar{M}_t$ is the mean curvature flow with initial condition $\bar{M}$ constructed in Lemma \ref{lemma embedded tube}. Let $H_{\bar{A}_t}$ and $H_{\bar{B}_t}$ be the mean curvatures of $\bar{A}_t$ and $\bar{B}_t$. Recall that $A_t$ and $B_t$ are rotationally symmetric about the $x_1$-axis, and recall that $\bar{\psi}$ is defined to respect rotational symmetry about the $x_1$ axis. The image of a rotationally symmetric surface under $\bar{\psi}$ will be rotationally symmetric with respect to $\gamma$. 
By a standard calculation of the mean curvatures of $\bar{A}_t, \bar{B}_t$, recalling that $|A|, |\nabla A| \leq C$ for $\gamma$, 
\begin{equation}
|H_{\bar{A}_t}(\psi(p), t) - H_{A_t}(p,t)| < K(n)C
\end{equation} 
for some dimensional constant $K(n)$. 
\medskip

By definition of $\bar{\psi}$, we have that
$$(1-K(n)C)|v| \leq |d\bar{\psi}(v)| \leq (1+K(n)C)|v|$$
where $\bar{\psi}$ is defined (recall that this depends on $r^*$).
Note that this estimate also depends on $r$, but this is suppressed, since $r$ can be chosen smaller than a fixed fraction of $1/L$.
\medskip

We may also assume without loss of generality that $H_{\bar{A}_t}(\psi(p),t) > 1$ for all points and times by choosing $r$ small enough, and we find $\delta_1>0$ depending on $C$ so that 
$$K(n)C < e^{-\delta_1}H_{\bar{A}_t}(\psi(p),t)$$

Now, we may show that $\bar{A}_t$ is a supersolution to mean curvature flow, after an appropriately small choice of $C, r^*, \delta$:
\begin{align*}
\Big|\frac{d\bar{A}_t}{dt}\Big| &= \Big|d\bar{\psi}\big(\frac{dA_t}{dt}\big)\Big|=  |d\bar{\psi}( e^{\delta}H_{A_t}\nu_{A_t})| \\
&> (1-K(n)C)|e^{\delta} H_{A_t}| |\nu_{A_t}| = (1-K(n)C)|e^{\delta}H_{A_t}|\\
&= (1-K(n)C)|e^{\delta}H_{\bar{A}_t}+e^{\delta}(H_{A_t} - H_{\bar{A}_t})|\\
&\geq (1-K(n)C)e^{\delta}\big||H_{\bar{A}_t}|- |H_{A_t} - H_{\bar{A}_t}| \big|\\
&\geq (1-K(n)C)e^{\delta}\big(|H_{\bar{A}_t}|-K(n)C\big)\\
\intertext{Applying the choice of $\delta_1$ as above,}
& > (1-K(n)C)e^{\delta}(1-e^{-\delta_1})|H_{\bar{A}_t}|\\
\intertext{For each $\delta>0$, we pick a small enough $C$, and hence a large enough $\delta_1$, such that $(1-K(n)C)e^{\delta}(1-e^{-\delta_1})\geq 1$ and so}
&> |H_{\bar{A}_t}|
\end{align*}

 Recalling that $\bar{A}_t$ is mean convex, this means that $\bar{A}_t$ is a supersolution for the flow. Thus, $\bar{A}_t$ is an inner barrier for $\bar{M}_t$ by Proposition \ref{subsolution proposition}. We may do the same for $\bar{B}_t$. Hence, for each $\delta>0$, we may find a small enough $C,r$ so that $\bar{A}_t, \bar{B_t}$ are a supersolution (resp. subsolution) and thus an inner (resp. outer) barrier for $\bar{M}_t$. 

$\medskip$

With these barriers in hand we now need to understand how they behave. Consider the following statement, which is immediate: 
\begin{prop} Let $T_r$ be the extinction time of the round cylinder of radius r. Then, for every $L>0$, $r \ll L$, and all $\eta_1\ll r$, there exists $r_1$ and $T^* = T^*(r, r_1, L, \eta_1) < T_r$ so that $M_t$ is $\eta_1$-close in the $C^2$ topology to a round sphere of radius $r_1$ by time $T^*$, with $\eta_1 \ll r_1$. 
\end{prop}
Since for all small $\delta$, we may find a $C>0$ such that $\bar{A}_t$ and $\bar{B}_t$ are inner and outer barriers, we may choose $\delta$ small enough to apply Proposition \ref{smooth} up until time $T^*$, as in the above proposition. This follows because a choice of $\delta$ small pinches the flow $\bar{M}_t$ between $\bar{A}_t$ and $\bar{B}_t$. A choice of $\delta$ gives us a choice of $C$ small, as described above. Proposition \ref{smooth} gives that the mean curvature flow $\bar{M}_t$ will flow, without singularities, to a hypersurface at time $T^*$ that is $\eta_1$-close in $C^2$ to a round sphere for some radius $r_1$ large relative to $\eta_1$. This means $\bar{M}_t$ has become convex, and so by Huisken's theorem the surface will continue to flow to a round point. 
$\medskip$

We end this argument with a discussion of which choices of parameters work. Normalizing $r$ to be one, as $L$ gets larger, $T^*$ approaches $T_1$. In particular, the curvature of $\bar{M}_t$ at the time it becomes convex becomes larger. The $\delta$ necessary to use Proposition \ref{smooth} depends on $\epsilon_0$ and $C$ from the Brakke regularity theorem as well as curvature bounds on the inner and outer barriers through time $T^*$. The curvature bounds on the inner and outer barriers through time $T^*$ are uniform as $\delta \to 0$, so we may always choose $\delta$ small enough to apply Proposition \ref{smooth}.
$\medskip$

Positive lower bounds on $\delta$ to ensure $\bar{A}_t$ and $\bar{B}_t$ are supersolutions and subsolutions gives a lower bounds on $C$ for which Theorem \ref{first theorem} holds. 
$\medskip$

Putting this together, if we fix $L$ we obtain an $r_1$, which then implies an upper bound on $\delta$ depending on both $r_1$ and the constants from the Brakke regularity theorem. The upper bound on $\delta$ then implies an upper bound $C$ on $|A|, |\nabla A|$ for which the construction above holds.

\section{Proof of Theorem \ref{second theorem}}
In this section, we prove Theorem \ref{second theorem}. Recall the notation set in Theorem \ref{second theorem}. In our notation, $M$ is a smooth surface, and $\widetilde{M}$ is another smooth surface given by gluing in a ``spike'' to $M$ at a point $ p \in M$. By a ``spike'' at $p$, we mean a suitable perturbation of a tubular neighborhood of a curve orthogonal to $T_p M$. The general idea of this theorem is that if $M \in \overline{\Sigma}$, i.e. that $M$ flows to a round point, then $\widetilde{M} \in \overline{\Sigma}$ as well. This will be proven by showing that the flow $\widetilde{M}_t$ will by some time be sufficiently close in $C^2$ to $M$, without developing singularities, which ensures that $\widetilde{M} \in \overline{\Sigma}$. This will be proven using localized barriers and applications of Proposition \ref{smooth} and Lemma \ref{refinement}. 
$\medskip$

We will begin by analyzing the model case of attaching a ``spike'' to an $n$-dimensional graph with bounded geometry. The following is a lemma that controls the flow of surfaces that are nearly graphical. Note that, for $n \geq 2$, we may obtain such graphs easily by attaching a two-convex tube as in Buzano, Haslhofer, and Hershkovits \cite{BuzanoHaslhoferHersh16} to a large, rotationally symmetric region of an extremely large sphere and smoothly extending that by an asymptotically flat graph. The reasoning for the choices made in the conditions of the following lemma will be made clear throughout the proof. In the following lemma, $B^n$ will denote a ball in the $n$-dimensional hyperplane $E$.

\begin{lem}\label{nearly graphical lemma}
Fix $L, \epsilon, C_1, C_2 > 0$. For each $\delta_2, r>0$, let $f = f(\delta_2, r)$ denote a smooth graph over an $n$-dimensional hyperplane $E \subset \mathbb{R}^{n+1}$ with the following properties:
\begin{enumerate}
\item The graph of $f$ satisfies $|A|^2 \leq C_1$, has entropy bounded by $C_2$, and is strictly mean convex in $B^n(0,\epsilon)$,
\item $f \geq 0$ in $B^n(0,\epsilon)$,
\item the graph of $f$ is rotationally symmetric around the $x_{n+1}$-axis, orthogonal to $E$, so that $f(x) = g(|x|)$,
\item $g(0)=L$ is the unique maximum point and $g(|x|)$ strictly decreases for $0<|x|<\epsilon$,
\item $g$ is $\delta_2$-close in $C^2$ norm to identically zero in $B(0,\epsilon) \setminus B(0,r)$.
\end{enumerate}
 Let $M \subset \mathbb{R}^{n+1}$ be a smooth hypersurface that is $\delta_1$-close in $C^2$ to the graph of $f = f(\delta_2, r)$. Then for each $\epsilon_0 > 0$, there exist $\delta_1, \delta_2, r$ (in practice, $\delta_1 \ll r\ll \delta_2$) with $\max\{\delta_1, \delta_2, r\} < \epsilon_0$ such that the flow $M_t$, with initial condition $M_0 = M$, exists for all time. Moreover, there exists $\delta$ and $T_0 = T_0(\delta_1, \delta_2, r)$ and $T_1 = T_1(C,\epsilon)$ such that for $t \in [T_0, T_1]$, $M_t$ is $\delta$-close in $C^2$ to the hyperplane $E$ in $B(0,\epsilon/4)$, and $\delta, T_0 \to 0$ as $\epsilon_0 \to 0$.

\end{lem}
\begin{proof}

Let $\Gamma(f)_t$ denote the mean curvature flow with initial condition $\Gamma(f)_0$, where $\Gamma(f)_0$ is the graph of $f = f(\delta_2, r)$ for $\delta_2, r>0$. 
$\medskip$

To prove this lemma, we may assume that $\delta_1 = 0$ without loss of generality. In other words, if for any $\epsilon_0>0$, there exists $\delta_2, r$ such that $\Gamma(f)_t$ satisfies the conclusion of this lemma, then there exists $\delta_1$ such that the flow $M_t$, with initial condition $M_0$ $\delta_1$-close in $C^2$ to $\Gamma(f)_0$, also satisfies the conclusions of this lemma. This follows from continuity of the flow over compact time intervals under $C^2$ perturbations of the initial condition. That is, there must exist $\delta_1>0$ such that if $M_0$ is $\delta_1$-close in $C^2$ to $\Gamma(f)_0$, then $M_t$ is $\delta$-close in $C^2$ to $\Gamma(f)_t$ for $t \in [0,T_1]$. Here, the $\delta$ and $T_1$ are as in the conclusion of the lemma for $\Gamma(f)_t$. Hence, assuming the $\delta_1=0$ case, there exists $\delta_1>0$ such that $M_t$ satisfies the conclusions of the lemma as well.  
$\medskip$
\newpage

\noindent \textbf{$\delta_1 = 0$ Case:}

It is sufficient to prove that for each $\epsilon_0>0$, there exist $0<\delta_2, r < \epsilon_0$ such that each $f = f(\delta_2, r)$ satisfies the conclusion of the lemma. Existence of the flow is immediate. Indeed, by Ecker-Huisken's interior estimates for graphical flows \cite{EckerHuisken89, EckerHuisken91}, the mean curvature flow $\Gamma(f)_t$ will exist for all time and will continue to be graphical and rotationally symmetric. We need to show that there exists $\delta>0$ and $T_0, T_1$ such that $\Gamma(f)_t$ is $\delta$-close in $C^2$ to $E$ in $B(0,\epsilon/2)$ for $t \in [T_0, T_1]$, where $\delta, T_0 \to 0$ as $\epsilon_0 \to 0$. There always vacuously exists some $\delta$ for any choice of $\delta_2, r$ and $\Gamma(f)_t$, so it is enough to show that we may find $\delta$ such that $\delta \to 0$ as $\epsilon_0 \to 0$. 
$\medskip$

The idea is to use a bowl soliton as a barrier to show that $\Gamma(f)_t$ quickly becomes $C^0$ close to the hyperplane $E$. Since the bowl soliton intersects $\Gamma(f)_0$, we must control the intersection in order to use the bowl soliton as a barrier. We use pseudolocality, mean convexity, and the Sturmian principle to control the intersection points so that the ``tip'' of the graph of $f$ must stay disjoint from the translating bowl soliton for a sufficient time. We then upgrade $C^0$ closeness to $E$ to $C^2$ closeness using one-sided minimization, the Brakke regularity theorem, and Ecker-Huisken estimates on the flow of graphs. 
$\medskip$

Let $f_t$ be the graph of the profile curve at time $t$. By Angenent-Altschuler-Giga~\cite{Altschuler1995}, the number of critical points of $f_t$ will not increase in time. This means there is a unique maximum point of $f_t$ at $0$ for all time. 
$\medskip$

For a given $\epsilon$, we see by arguments as in \cite{Mramor} that there will be a period of time $[0,T_2]$ so that $\Gamma(f)_t$ will remain mean convex within the ball $B(0, \epsilon/2)$. We stress that $T_2$ is uniformly bounded from below for all sufficiently small $r, \delta_2$. 
$\medskip$

Specifically, this follows by an application of the strong maximum principle and pseudolocality. We may apply pseudolocality in $B(0,\epsilon) \setminus B(0,\epsilon/4)$ so that $\Gamma(f)_t$ remains strictly mean convex on the boundary of $B(0,\frac{\epsilon}{2})$ for $[0,T_2]$. By the strong maximum principle, $\Gamma(f)_t$ will remain strictly mean convex in $B(0,\epsilon/2)$. Similarly, by pseudolocality, we have that $f_t\geq 0$ in $B^n(0,\frac{\epsilon}{2})$ for $t \in [0,T_2]$.
$\medskip$

\noindent \textbf{$C^0$-close:}

Now, we will show that for any choice of $T^*_0$, $\delta^*$ we may choose $\delta_2, r< \epsilon_0$ sufficiently small so that $\Gamma(f)_t$ is $\delta^*$-close in $C^0$ to the hyperplane $E$ in $B(0,\epsilon/2)$ by time $T^*_0$. 
$\medskip$

We will prove this by comparison with a bowl soliton $\mathcal{B}$, which is a convex rotationally symmetric solution to mean curvature flow which translates at a constant speed depending on a scale factor. We may place the bowl soliton $\mathcal{B}$ so that it is graphical over the hyperplane $E$, rotationally symmetric about the $x_{n+1}$-axis, and so that it initially intersects $E$ outside $B^n(0,2r)$. Since $\Gamma(f)_t$ is $\delta_2$-close in $C^2$ to $E$ outside $B^n(0,2r)$, $\mathcal{B}$ intersects $\Gamma(f)_t$ outside $B^n(0,r)$, taking $\delta_2$ small enough relative to $r$. See Figure \ref{figure 6}. As a graph, the bowl soliton $\mathcal{B}$ has a unique maximum at the origin. The scale of the bowl soliton will be chosen sufficiently small later in the argument. That is, we may scale the bowl soliton $\mathcal{B}$ so that $\mathcal{B}$ intersects $E$ in a sphere of radius $s \in (2r, \frac{\epsilon}{2})$. A key point is that if $v$ is the speed of $\mathcal{B}$ as a translator, $v \to \infty$ as $s \to 0$. Now, we will use $\mathcal{B}_t$, the flow of $\mathcal{B}$, as a barrier for $\Gamma(f)_t$. Since both $\Gamma(f)_t$ and $\mathcal{B}_t$ are rotationally symmetric, we may apply the Sturmian principle of Angenent-Altschuler-Giga~\cite{Altschuler1995}. The Sturmian principle says that the number of intersections of the profile curves of rotationally symmetric hypersurfaces does not increase. Let the graph $b_t$ be the profile curve of $\mathcal{B}_t$.  Let $f_t$ denote the profile curve of the rotationally symmetric graph $f(\cdot, t)$. Note that $b_t, f_t$ are one-dimensional profile curves defined over an $x$-axis whose rotation around the $y$-axis give $\mathcal{B}_t, \Gamma(f)_t$, respectively.  Since $\Gamma(f)_t$ is mean convex in $B^n(0,\frac{\epsilon}{2})$ and $f_t$ remains graphical for all time, $f_t$ will be strictly decreasing in $[-\frac{\epsilon}{2}, \frac{\epsilon}{2}]$ until time $T_2$. 
$\medskip$

Let $x_1(0),x_2(0)$ denote the values on the $x$-axis of the two intersection points between $f_0$ and $b_0$. Let $x_1(t),x_2(t)$ denote the intersection points (if they exist) between $f_t$ and $b_t$ for later times $t>0$. By the Sturmian principle, the number of intersection points does not increase and only decreases at double zeros. So, the intersection points only fail to exist after the first time that $x_1(t) = x_2(t)$. By construction, $f_t(p) \leq b_t(p)$ for $p \in [x_1(t), x_2(t)]$ so long as $x_1(t) \neq x_2(t)$. Let $\bar{x}_1(t), \bar{x}_2(t)$ be the $x$-coordinates of the intersection points between $b_t$ and $f_0$ outside the interval $[-2r, 2r]$, for as long as there exist such intersections. Note that there will be other intersections of $b_1$ with $f_0$ for some $t>0$, but there are only two intersection points outside $[-2r, 2r]$ initially. Choose $r$ and the scale of the bowl soliton small enough so that there is no intersection between $b_t, f_0$ before time $t=T_2$. Then, if $\bar{x}_1(t), \bar{x}_2(t)$ exist and $\bar{x}_1(t) \neq \bar{x}_2(t)$, then $b_t$ has two intersection points with $f_t$ outside $[-2r, 2r]$ (i.e. $x_1(t), x_2(t)$ exist so $|x_1(t)|, |x_2(t)|>2r$) and $x_1(t) \neq x_2(t)$. This follows since $f_t$ is decreasing in $[-\frac{\epsilon}{2}, \frac{\epsilon}{2}]$ for $t \in [0,T_2]$ due to mean convexity, so if $b_t$ intersects $f_0$ at $\bar{x}_1(t), \bar{x}_2(t)$, there will exist intersection points $x_1(t) \neq x_2(t)$ between $f_t$ and $b_t$ in $[-\frac{\epsilon}{2}, \frac{\epsilon}{2}]$. Thus, $f_t(p) < b_t(p)$ for $p \in [-2r, 2r]$ for as long as $\bar{x}_1(t), \bar{x}_2(t)$ exist and $\bar{x}_1(t) \neq \bar{x}_2(t)$.
\medskip

\begin{figure}
\centering
\includegraphics[scale = .45]{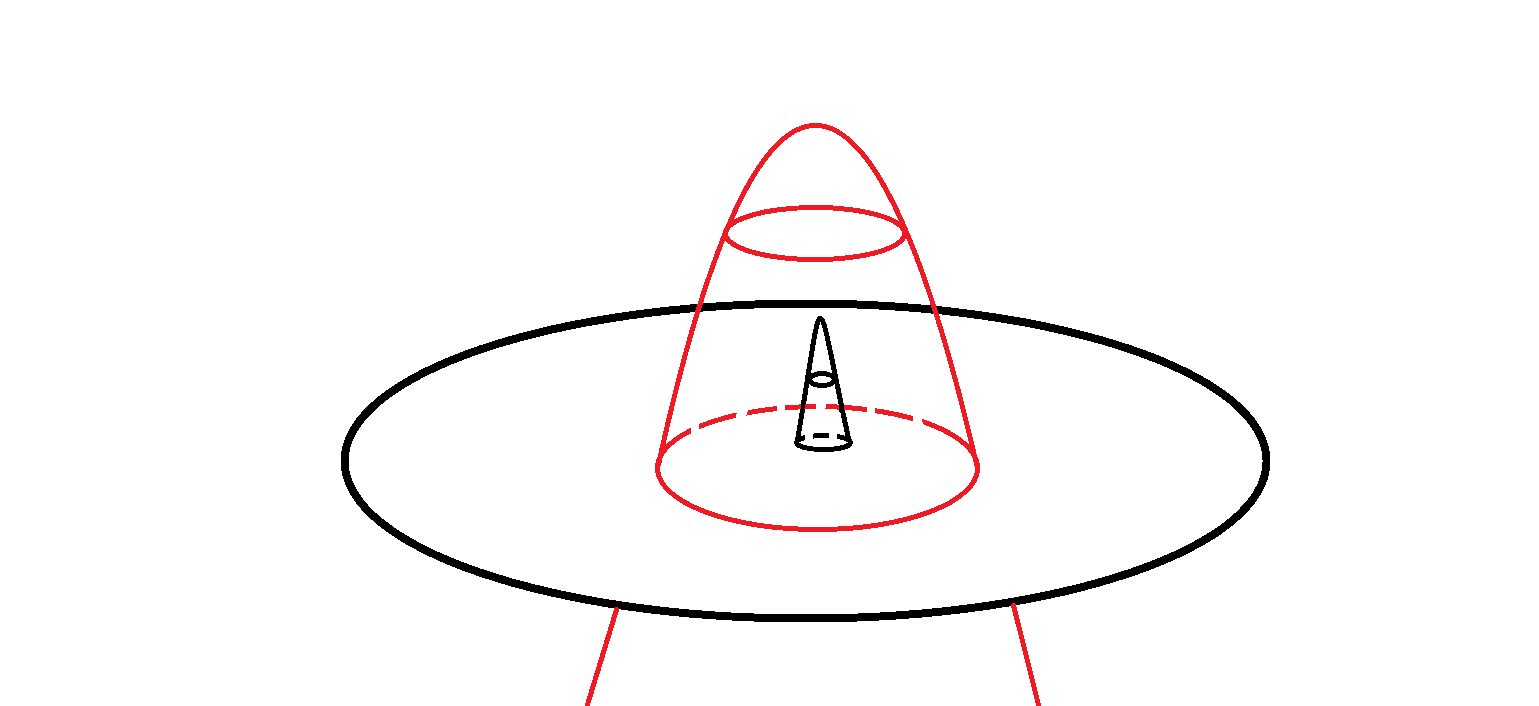}
\caption{}
\label{figure 6}
\end{figure}

By construction, $f_0(p) \leq \delta_2$ for $p \in [-\frac{\epsilon}{2}, \frac{\epsilon}{2}]\setminus [-2r, 2r]$. Since $f_t$ is decreasing in $[-\frac{\epsilon}{2}, \frac{\epsilon}{2}]$ for $t \in [0,T_2]$, $f_t(p) \leq \delta_2$ for $p \in [-\frac{\epsilon}{2}, \frac{\epsilon}{2}]\setminus [-2r, 2r]$. Pick $\delta^*>0$. We may then choose $r, \delta_2$ small enough relative to the scale of the bowl soliton $b_t$ so that $b_t$ intersects $f_0$ in $[-\frac{\epsilon}{2}, \frac{\epsilon}{2}]\setminus [-2r, 2r]$ until the time $T_0^* < T_2$ such that $b_{T_0^*}(0) =\delta^*$. Indeed, this follows from the fact that $f_0$ is $\delta_2$-close in $C^2$ to $0$ in $[-\epsilon, \epsilon]\setminus [-r, r]$, and the profile curve of the bowl soliton $b_t$ is strictly decreasing with a unique maximum at $0$. Recall that $b_t$ translates over the flow downwards at a constant speed $v$, with $v \to \infty$ as the scale $s \to 0$. Thus, $\bar{x}_1(t), \bar{x}_2(t)$ exist with $\bar{x}_1(t) \neq \bar{x}_2(t)$ for $t \in [0,T_0^*]$. By the argument in the previous paragraph, this ensures that $f_t(p) \leq b_t(p) \leq b_t(0) \leq \delta^*$ for $p \in [-2r, 2r]$. Thus, for each $\delta^*, T_0^*$, there exists $\delta_2, r< \epsilon_0$ small such that $f_t\leq \delta^*$ in $[-\frac{\epsilon}{2}, \frac{\epsilon}{2}]$ by time $T_0^* < T_2$. 
$\medskip$

For given $\delta_2, r < \epsilon_0$ sufficiently small, denote by $T_0^*(\delta^*)$ the first time before $T_2$ (which we recall is uniformly controlled) such that $\Gamma(f)_t$ is $\delta^*/2$-close in $C^0$ to the hyperplane $E$ in $B(0,\epsilon/2)$. We have that this exists by the discussion in the previous paragraph. 
$\medskip$

Since $\Gamma(f)_t$ is $\delta^*/2$-close in $C^0$ to $E$ in $B(0,\epsilon/2)$ by time $T_0^*(\delta^*)$, by the avoidance principle, there exists $T_1 = T_1(C_1,\epsilon)$ such that $\Gamma(f)_t$ is $\delta^*$-close in $C^0$ to $E$ in $B(0,\epsilon/2)$ for $t \in [T_0^*(\delta^*), T_1]$. 
$\medskip$

\noindent \textbf{$C^2$-close:}

Now, for each sufficiently small $\epsilon_0$ and a choice of $\delta_2, r < \epsilon_0$, we will find $T_0 \in [T_0^*(\delta^*), T_1]$ and $\delta$ such that $\Gamma(f)_t$ is $\delta$-close in $C^2$ to the hyperplane $E$ in $B(0,\epsilon/2)$ for $t \in [T_0, T_1]$, and $\delta, T_0 \to 0$ as $\epsilon_0 \to 0$. This will follow from an application of the one-sided minimization theorem and Brakke regularity (Theorem \ref{Brakke}) to find uniform curvature bounds, combined with $C^0$-closeness and an application of Lemma \ref{refinement} (see also \cite{Mramor}). Alternately we could use that such a graph must be close in density to a plane on large scales and apply monotonicity as in the proof of Theorem \ref{third theorem}, but the following argument applies generally to mean convex graphs. 
$\medskip$

By White's one-sided minimization, Theorem \ref{theorem one sided min} (using the fact that $\Gamma(f)_t$ is  mean convex for $t \in [0,T_1]$), and the $\delta^*$-closeness in $C^0$ found above, 
$$\text{Area}(\Gamma(f)_t \cap B(x,r)) \leq (1+2n\delta^*/r)\omega_n r^n$$
for $x \in E \cap B(0,\epsilon/4) \subset \mathbb{R}^{n+1}$, $r \in (0, \epsilon/4]$, and $t \in [T_0^*(\delta^*), T_1]$. Define 
$$A(t, r):= \text{Area}(\Gamma(f)_t \cap B(x,r))$$
Applying a standard integration by parts, for $x \in \Gamma(f)_t \cap B(0,\epsilon/4)$ and $t \in (T_0^*(\delta^*)+r_*^2, T_1]$, with $r^*>0$ to be chosen later, 
\begin{align*}
\Theta(\Gamma(f)_t, x, r_*)  &= \frac{1}{(4 \pi r_*^2)^{n/2}} \int_{s=0}^{\infty} \frac{\omega_n s^{n+1}}{2r_*^2} e^{\frac{-s^2}{4r_*^2}}\frac{A(t-r_*^2, s)}{\omega_n s^n}\,ds\\
&= \frac{1}{(4 \pi r_*^2)^{n/2}} \int_{s=0}^{\epsilon/4} \frac{\omega_n s^{n+1}}{2r_*^2} e^{\frac{-s^2}{4r_*^2}}\frac{A(t-r_*^2, s)}{\omega_n s^n}\,ds\\ 
&\quad\,+\frac{1}{(4 \pi r_*^2)^{n/2}} \int_{s=\epsilon/4}^{\infty} \frac{\omega_n s^{n+1}}{2r_*^2} e^{\frac{-s^2}{4r_*^2}}\frac{A(t-r_*^2, s)}{\omega_n s^n}\,ds\\
\intertext{Since $\Gamma(f)_0$ has entropy initially bounded by $C_2$, the second term can be bounded by $C(r_*, \epsilon, C_2)$, which goes to zero as $r_* \to 0$. Pick $r_*$ such that $C(r_*, \epsilon, C_2)< \epsilon_0/3$, where $\epsilon_0$ is the constant from Brakke regularity (Theorem \ref{Brakke}), and let $s^* = \frac{6n\delta^*}{\epsilon_0}$.}
&\leq \frac{1}{(4 \pi r_*^2)^{n/2}} \int_{s=0}^{\epsilon/4} \frac{\omega_n s^{n+1}}{2r_*^2} e^{\frac{-s^2}{4r_*^2}}\frac{A(t-r_*^2, s)}{\omega_n s^n}\,ds + C(r_*, \epsilon, C_2)\\
&\leq \frac{1}{(4 \pi r_*^2)^{n/2}} \int_{s=0}^{\epsilon/4} \frac{\omega_n s^{n+1}}{2r_*^2} e^{\frac{-s^2}{4r_*^2}}\frac{A(t-r_*^2, s)}{\omega_n s^n}\,ds + \epsilon_0/3\\
&\leq \frac{1}{(4 \pi r_*^2)^{n/2}} \int_{s=0}^{s_*} \frac{\omega_n s^{n+1}}{2r_*^2} e^{\frac{-s^2}{4r_*^2}}\frac{A(t-r_*^2, s)}{\omega_n s^n}\,ds\\
&\quad + \sup_{s \in [s^*, \epsilon/4]}\frac{A(t-r_*^2, s)}{\omega_n s^n}\frac{1}{(4 \pi r_*^2)^{n/2}} \int_{s=0}^{\infty} \frac{\omega_n s^{n+1}}{2r_*^2} e^{\frac{-s^2}{4r_*^2}}\,ds + \epsilon_0/3\\
\intertext{Applying the bound on area ratios, which holds for $\Gamma(f)_{t-r_*^2}$ for $t \in (T_0^*(\delta^*)+r_*^2, T_1]$,}
&\leq 1+2\epsilon_0/3 + \frac{1}{(4 \pi r_*^2)^{n/2}} \int_{s=0}^{s_*} \frac{\omega_n s^{n+1}}{2r_*^2} e^{\frac{-s^2}{4r_*^2}}\frac{A(t-r_*^2, s)}{\omega_n s^n}\,ds\\
\intertext{Applying the definition of $s_*$ and the entropy bound, we may bound the second term by $C(\delta^*, C_2)$, which goes to zero as $\delta^* \to 0$.}
&\leq 1+ 2\epsilon_0/3 +  C(\delta^*, C_2)
\end{align*}

By the argument for $C^0$ closeness, for any $\delta^*$, we may choose parameters $\delta_2, r$ such that $\Gamma(f)_t$ is $\delta^*$-close in $C^0$ to the hyperplane $E$ for $t \in [T_0^*(\delta^*), T_1]$. This means that we may choose parameters sufficiently small such that $C(\delta^*, C_2) < \epsilon_0/3$. Thus, by the calculation above, for $x \in \Gamma(f)_t \cap B(0,\epsilon/4)$ and $t \in (T_0^*(\delta^*)+r_*^2, T_1]$,
$$\Theta(\Gamma(f)_t, x, r_*) < 1+\epsilon_0$$
Note that we may choose $r_*, \delta^*$ arbitrarily small, resulting in a possibly smaller choice of parameters $\delta_2, r$. By Huisken monotonicity, we have that $\Theta(\Gamma(f)_t, x, r)<1+\epsilon_0$ for all $r \in (0, r^*]$. We may then apply Brakke regularity (Theorem \ref{Brakke}) to find that for $x \in \Gamma(f)_t \cap B(0,\epsilon/4)$, $\Gamma(f)_t$ satisfies
\begin{equation}\label{equation A bound}|A|(x,t) \leq \frac{C(n)}{r_*}\end{equation}
for $t \in (T_0^*(\delta^*)+r_*^2, T_1]$. 
$\medskip$

Now, for each $r_*, \epsilon^*\ll 1$, there exists $\delta^*(r_*), \delta_2, r$ such that $\Gamma(f)_t$ is $\epsilon^*$-close in $C^1$ to $E$ in $B(0,\epsilon/4)$ for $t \in (T_0^*(\delta^*(r^*))+r_*^2, T_1]$. This follows from choosing $\delta_2, r$ to find $\delta^*(r^*)$-closeness in $C^0$ to $E$ in $B(0,\epsilon/4)$ for $t \in  (T_0^*(\delta^*), T_1]$ combined with (\ref{equation A bound}), a uniform $C^2$ bound on $\Gamma(f)_t$ over the hyperplane $E$ in $B(0,\epsilon/4)$. Since the parameters $\delta_2, r$ may be chosen so that $\Gamma(f)_t$ is arbitrarily $\epsilon^*$-close in $C^1$ to $E$ in $B(0,\epsilon/4)$ for $t \in  (T_0^*(\delta^*(r^*))+r_*^2, T_1]$, we find that  for $x \in \Gamma(f)_t \cap B(0,\epsilon/4)$ and $t \in (T_0^*(\delta^*(r^*))+2r_*^2, T_1]$, $\Theta(\Gamma(f)_t, x, r_*) \to 1$ as $\delta_2, r \to 0$. 
$\medskip$

By Huisken monotonicity and Lemma \ref{refinement} (applied to the spacetime cylinder $B(0,\epsilon/4)\times \mathbb{R}$), we have that for $x \in \Gamma(f)_t \cap B(0,\epsilon/4)$ and $t \in (T_0^*(\delta^*(r^*))+2r_*^2, T_1]$, $|A|^2(x,t) \to 0$ as $\delta_2, r \to 0$. Combined with pseudolocality applied outside $B(0,\epsilon/4)$, we find the analogous statement for $C^2$ norm, which completes the proof of the lemma. The fact that $\delta, T_0 \to 0$ as $\delta_2, r \to 0$ (that is, as $\epsilon_0 \to 0$) follows using the fact that $T_0^*(\delta^*(r^*))\to 0$ as $\delta^*(r^*)\to 0$.

\end{proof}

Lemma \ref{nearly graphical lemma} above gives us a local model, which we will use to prove Theorem \ref{second theorem}. 
\begin{rmk}\label{remark 4.1}In Lemma \ref{nearly graphical lemma}, if we fix $R, C_1, C_2$, then we may in fact find $\delta_2>0$ such that for each $L>0$, there is $0<\delta_1, r \ll 1$ such that the conclusion of the lemma holds. The $\delta, T_0$ that are found in Lemma \ref{nearly graphical lemma} are independent of $L$ so long as $\delta_1, r$ are possibly chosen smaller. That is, varying the choice of $L$ in Lemma \ref{nearly graphical lemma} only affects our choice of $r, \delta_1$. This follows immediately from the proof of Lemma \ref{nearly graphical lemma}.
\end{rmk}


\noindent \textbf{Construction of $\widetilde{M}$:}

First we set notation and explain how to construct the desired initial condition $\widetilde{M}$ in Theorem \ref{second theorem}. Let $M$ be a surface in $\overline{\Sigma}$ (the set of surfaces which shrink to round points) and fix $p \in M$ and $L>0$. Fix a segment $\gamma$ satisfying the hypotheses of Theorem \ref{second theorem} with $r, \epsilon\ll 1$ to be chosen. Fix $N \gg 1$. Recall from the statement of Theorem \ref{second theorem} that we chose $N\gg 1$ at the outset to control how close the error between the spike and the tubular neighborhood of radius $r$ around $\gamma$. Let $T_{\epsilon}(\gamma)$ be the solid $\epsilon$-radius tubular neighborhood of $\gamma$. For any $r \ll 1$, let $\partial T_r(\gamma)$ denote a smooth surface which is $\frac{r}{N}$-close in $C^1$ to the boundary of the radius $r$ tubular neighborhood of $\gamma$ and which is graphical over $T_p M$ ($N$, as in the statement of the theorem is used to control how close $\widetilde{M}$ is to the tubular neighborhood of $\gamma$). Our construction of $\widetilde{M}$ uses Buzano-Haslhofer-Hershkovits' construction of two-convex gluings of tubular neighborhoods of curves  to two-convex surfaces (see \cite[Theorem 4.1]{BuzanoHaslhoferHersh16}). 
$\medskip$

The idea to construct $\widetilde{M}$ is to rescale $M$ around the point $p$ so that it is nearly flat, perturb the rescaled $M$ to be two-convex around $0$, and apply Buzano-Haslhofer-Hershkovits' construction around $0$. Rescale $M$ around $p$, moving $p$ to the origin $0$, and provisionally call the rescaled and perturbed surface $M^*$. We will smoothly perturb $M^*$, relabelling the perturbed surface $M^*$, so that $M^*$ is rotationally symmetric in $B(0,10)$ and has a sign on mean curvature in $B(0,2)$. Recall that $\gamma$ will be rescaled to a longer line segment $\gamma^*$ normal to $T_0 M^*$.
$\medskip$

Pick $\delta>0$ small enough so that if $\widetilde{M}_t$ is $\delta$-close in $C^2$ to $M$, then $\widetilde{M}_t \in \overline{\Sigma}$, i.e. $\widetilde{M}_t$ will flow smoothly to a round point. This is possible due to $\overline{\Sigma}$ being open in $C^2$ due to continuous dependence of the flow on smooth perturbations of the initial data. Note that if $\widetilde{M^*}_t$ is $\delta$-close in $C^2$, then by scaling, $\widetilde{M}_t$ is $\delta$-close in $C^2$ to $M$ as well. In Lemma \ref{nearly graphical lemma}, $\delta \to 0$ as $\delta_2, r \to 0$. Using Lemma \ref{nearly graphical lemma} (see Remark \ref{remark 4.1}) with a choice of $\epsilon, C_1=1$ and $C_2$ the entropy of $M$, we find $\delta_2>0$ such that for each $L>0$, there is $\delta_1, r>0$ such that the $M_t$ of Lemma \ref{nearly graphical lemma} is $\delta/2$-close in $C^2$ to $E$ in $B(0,1/2)$. Given this choice of $\delta_2$, rescale and perturb $M$ to $M^*$ so that $M^*$ is rotationally symmetric in $B(0,10)$, $M^*$ a sign on mean curvature in $B(0,2)$, and $M^*$ is $\delta_2$-close in $C^2$ to the hyperplane $T_0 M^*$ in $B(0,10)$.
$\medskip$

Then, applying Buzano-Haslhofer-Hershkovits' construction in $B(0,1)$, using the $r$ we found above, there exists a surface $\widetilde{M^*}$ such that $\widetilde{M^*} = M^* \cup \partial T_r(\gamma^*)$ (in the sense of Theorem \ref{second theorem}) and such that $\widetilde{M^*}$ is graphical over $T_0 M^*$ in $B(0,2)$ and satisfies the assumptions of the graph $f= f(\delta_2, r)$ of Lemma \ref{nearly graphical lemma} in $B(0,2)$. Rescaling $\widetilde{M^*}$ back to the scale of $M$, we find for sufficiently small $r, \epsilon \ll 1$ a (possibly immersed) surface $\widetilde{M}$ such that $\widetilde{M}  = M \cup \partial T_r(\gamma)$ outside $B(p,\epsilon)$ and $\widetilde{M} \cap T_{r}(\gamma)$ may be written as a graph over $B^n(p,\epsilon)$ with a sign on mean curvature, where $B^n(p,\epsilon) = T_p M \cap B(p,\epsilon)$ is a ball in the subspace $T_p M$. By construction, $\widetilde{M}$ will be $2$-convex, with respect to the inward pointing normal of $\partial T_r(\gamma)$, inside $B(p,\epsilon)$. We note that if $\gamma$ is oriented so that it points inward for $M$, then $\widetilde{M}$ will have negative mean curvature in $T_{R}(\gamma)$. This shows that the mean curvature of $\widetilde{M}$ will have a sign inside $T_{\epsilon}(\gamma)$. The picture to have in mind regarding the construction of $\widetilde{M}$ is Figure \ref{figure 7}.
$\medskip$

\begin{figure}
\centering
\includegraphics[scale = .45]{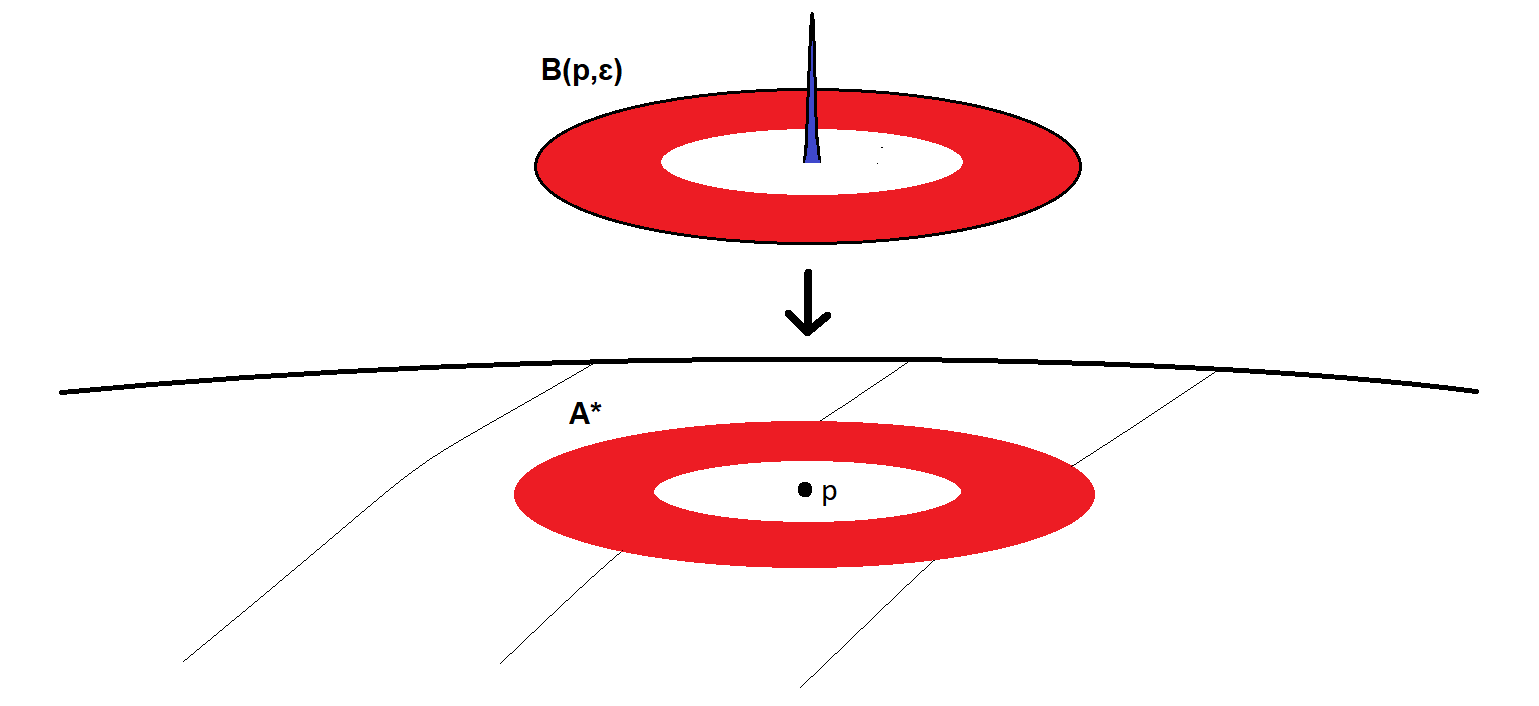}
\caption{}
\label{figure 7} 
\end{figure}

\noindent\textbf{Conclusion of proof:}

We will prove $\widetilde{M}$ satisfies the conclusions of this theorem for small enough choices of parameters $r, \epsilon$ (note that $\epsilon$ corresponds to the parameter $\delta_2$ for $M^*$). Recall that $\overline{\Sigma}$ is the set of smooth closed surfaces which flow to round points. By our choice of $\delta$, if $\widetilde{M}_t$ is $\delta$-close in $C^2$ to $M$, then $\widetilde{M}_t \in \overline{\Sigma}$. The conclusion of Theorem \ref{second theorem} follows from proving that, for $\widetilde{M^*}$ constructed above, there exists $T_0>0$ such that $\widetilde{M^*}_{T_0}$ is $\delta$-close in $C^2$ to $M^* \in \overline{\Sigma}$. It is enough to prove that there exists $T_0>0$ such that $\widetilde{M^*}_{T_0}$ is $\delta$-close in $C^2$ to the hyperplane $T_0 M^*$ in $B(0,1/4)$ for small enough choices of $r, \delta_2$ in the construction of $\widetilde{M^*}$. This is sufficient since pseudolocality ensures that $\widetilde{M^*}_{T_0}$ will be $\delta$-close in $C^2$ to $M ^*$ in the complement of $B(0,1/4)$ for $t \ll 1$. 
$\medskip$

We let $F$ be the graph defined over $B^n(0, 2)$ which coincides with $\widetilde{M^*}$. We will extend $F$ to graphs $F^+, F^-$ defined over all of $T_0 \widetilde{M^*}$, such that $F^+, F^-$ satisfy the assumptions of Lemma \ref{nearly graphical lemma}. We will then use $F^+, F^-$ as barriers in Proposition \ref{smooth} to control the flow $\widetilde{M^*}_t$. Consider an annulus $A^* = A^*(2, 3)$ in $T_0 \widetilde{M^*}$ (which corresponds to $A^*(2\epsilon, 3 \epsilon)$ without rescaling), and let $\nu(p)$ be the unit normal to $M^*$ at $0$.

$\medskip$

Then, consider the shifts $F \pm h \,\nu(p)$ by a small distance $h$, to be chosen later. We will extend $F \pm h \, \nu(p)$ to graphs  $F^+, F^-$ defined over $T_0 \widetilde{M^*}$ with bounded curvature and entropy. Modify $F^+, F^-$ in $A^*$ such that $F^+, F^-$ are disjoint in $A^*$ and such that $|F^+|, |F^-|$ is radially increasing over $B(3, 10)$ with $|F^+|, |F^-|>1$ over $T_0\widetilde{M^*}\setminus B(0,10)$. That is, $F^+, F^-$ are slight shifts of $F$ which are ``flared,'' as in Figure \ref{figure 8}. By design, the graphs $F^+, F^-$ do not intersect over all of $T_0 \widetilde{M^*}$. The graphs $F^+, F^-$ give domains $I_t$ and $O_t$ (whose boundaries are the graphs of $F^+, F^-$, respectively) to which we will apply Proposition \ref{smooth}. We will choose $h$ small enough to apply Proposition \ref{smooth}. So, $\partial I_t, \partial O_t$ do not intersect. 
$\medskip$

By pseudolocality and the comparison principle applied in $B(0,2)$, there exists a time $\widetilde{T}$, indepedent of $h, \delta_2, r$, such that $\partial I_t, \partial O_t$ and $\widetilde{M^*}_t$ are disjoint (as long as $\widetilde{M^*}_t$ exists) in $B(0,2)$ for $[0, \widetilde{T}]$. The setup is encapsulated in Figure \ref{figure 8}, where $h$ is drawn exaggeratedly large.

\begin{figure}[H]
\centering
\includegraphics[scale = .42]{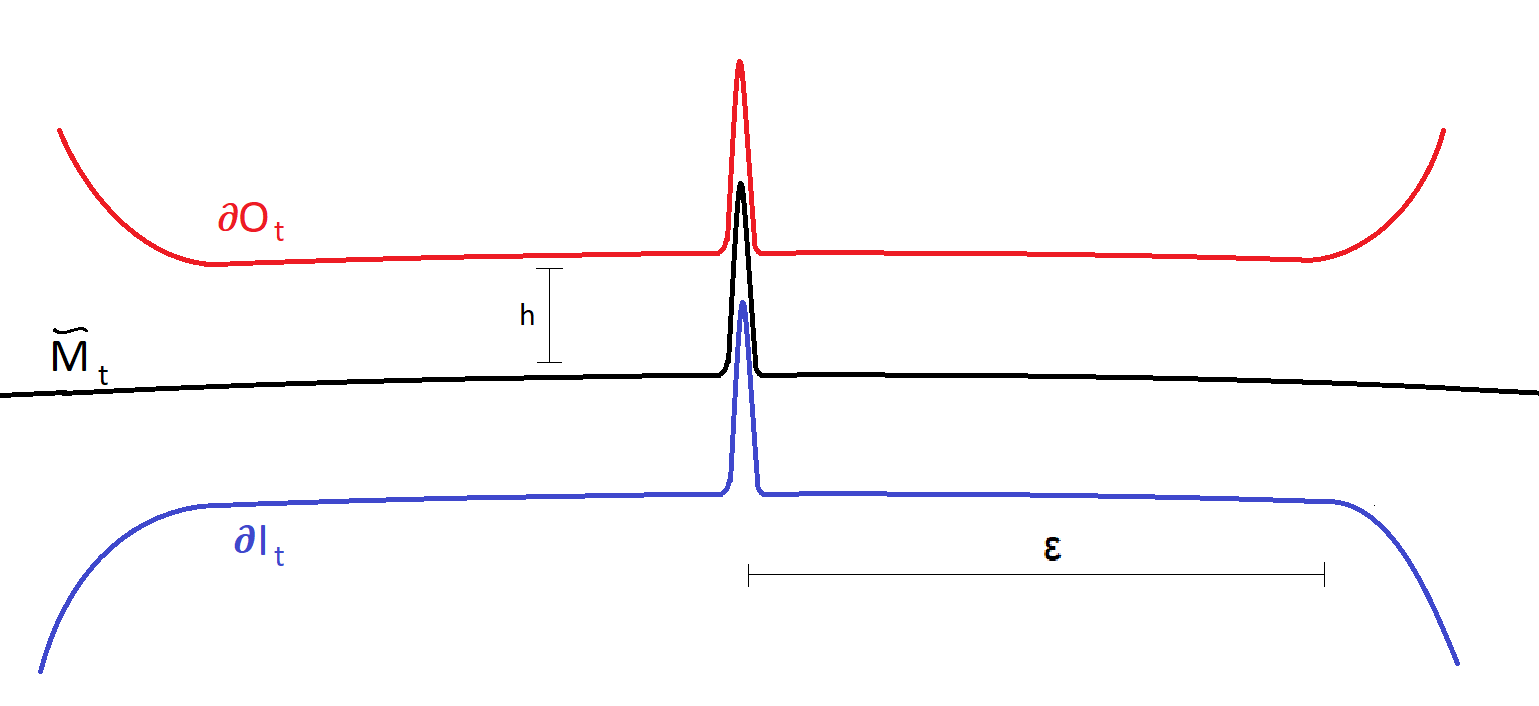}
\caption{}
\label{figure 8}
\end{figure}

Now, we will show that $\widetilde{M^*}_t$ will flow without singularities by applying Proposition \ref{smooth}. Proposition \ref{smooth} shows that for some $t \ll 1$, $\widetilde{M^*}_t$ will be $\delta$-close in $C^2$ to $M^*$ in $B(0,1/4)$. By construction, we may apply Lemma \ref{nearly graphical lemma} to the graphs of $F^+, F^-$, i.e. $\partial I_t$ and $\partial O_t$. This means that $\partial I_t, \partial O_t$ will be $\delta/2$-close in $C^2$ to the hyperplane $T_0 \widetilde{M^*}$ in $B(0,1/2)$ by time $T_0$. We may choose $h$ small enough to apply Proposition \ref{smooth} in $B(0,1/2)$. Moreover, by Lemma \ref{nearly graphical lemma}, we may take parameters $\delta_2, r$ small enough so that $T_0 \ll \widetilde{T}$ (taking the parameter $r$ smaller means $h$ needs to be taken smaller, but $\widetilde{T}$ is independent of $h$). This means that $\partial I_t$ and $\partial O_t$ form inner and outer barriers for the flow $\widetilde{M^*}_t$ in $B(0,1/2)$ and by Lemma \ref{nearly graphical lemma}, $\widetilde{M^*}_t$ will flow smoothly and will be $\delta$-close in $C^2$ to $M^*$ in $B(0,1/4)$ by time $T_0$.
$\medskip$

Note also that this application of Proposition \ref{smooth} shows that $\widetilde{M^*}_t$ will be $\delta$-close in $C^2$ to the hyperplane $T_0 M^*$ in $B(0,1/4)\setminus B(0,1/8)$ for $t \in [0,T_0]$. By pseudolocality, we have that for $t \in [0,T_0]$, outside $B(0, 1/4)$, $\widetilde{M^*}_t$ flows smoothly and is $\delta$-close in $C^2$ to $M^*\setminus B(0,1/4)$. To emphasize this use of pseudolocality, we have the following:
\begin{lem}\label{pl} Let $\widetilde{M}$, as defined above, satisfy $|A|^2 < C$ on $\widetilde{M} \setminus B(p, \epsilon)$. For each $0<C^*$, there exists $0 <T^*=T^*(C^*, \epsilon) < T$ such that $|A|^2 < (1+C^*)C$ on $\widetilde{M}_t \setminus B(p, \epsilon)$ for $t \in [0, T^*]$.
\end{lem} 

We choose $\epsilon$ and $r$ such that the time $T_0$ obtained from Lemma \ref{nearly graphical lemma} is smaller than $T^*$ obtained from Lemma \ref{pl}. That way, the previous discussion ensures that by time $T_0$, $\widetilde{M}_{T_0}$ will be $\delta$-close in $C^2$ and graphical over $M$ in $B^n(p,\epsilon/2)$, and Lemma \ref{pl} ensures that $\widetilde{M}_{T_0}$ will be close in $C^2$ to $M$ after possibly choosing $T_0$ smaller than $T$ from Lemma \ref{pl}. As discussed before, this implies that $\widetilde{M}_t \in \overline{\Sigma}$. Indeed, since $M$ is a surface in the interior of $\overline{\Sigma}$ and since $\widetilde{M}_{T_0}$ is $\delta$-close to $M$ in $C^2$ (depending on a sufficiently small choice of $\epsilon, r$), $\widetilde{M}_t$ will flow smoothly until it shrinks to a round point.
$\medskip$

\noindent \textbf{Iteration:}

Now, we see that we may iterate this construction. Since $\widetilde{M}_t \in \overline{\Sigma}$, we may apply Theorem \ref{second theorem} to $\widetilde{M}$ in place of $M$. That is, any surface $\widetilde{M}$ constructed by the above procedure must lie in $\overline{\Sigma}$, which is open in $C^2$, meaning that all small enough $C^2$ perturbations of $\widetilde{M}$ must shrink to round points as well. This is what allows for this construction to be iterated. We may also make all modifications of $M$ simultaneously as well, i.e. attach any finite number of spikes simultaneously using the above construction. Suppose we have already constructed $\widetilde{M}_t$ so that it contains a spike as above. Then, for any other choice of point $p'$, we may attach a new $\gamma'$ at $p'$. By picking $\epsilon'$ and $r'$ small enough for $\gamma'$, we may ensure that there is a time $T_0'$ such that the flow produces a perturbation of $M$ that is within the error $\delta_1$ used in Lemma \ref{nearly graphical lemma} by time $T_0'$. Thus, at time $T_0'$, the flow will look like an admissible perturbation of $\widetilde{M}$ in the sense of Lemma \ref{nearly graphical lemma} and so it must then shrink to a round point as $\widetilde{M}$ does. 
$\medskip$

\section{Proof of Theorem \ref{third theorem}} 
This proof follows the general idea of Theorem \ref{second theorem}. Unlike Theorem \ref{second theorem} though, the $\widetilde{M}$ constructed in this theorem is possibly not mean convex. To account for this difference, we will use the rotational symmetry of $\widetilde{M}$, and the subsequent local control we gain on it from Sturmian theory. In fact with suitable modifications, the approach of this proof provides an alternative method of proof for Theorem \ref{second theorem} as well. 
\medskip

As in the previous section, we arrange $\widetilde{M}$ so that, after some small later time, $\widetilde{M}_t$ is as close as we wish to the original surface $M$ in $C^2$ norm. To proceed we will modify the profile curve $f$ of $M$ in its $(\delta, c)$-cylindrical domain by L-shaped curves, suitably capped, as indicated in Figure \ref{Figure 10} (with the definition of $(\delta,c )$ cylindrical given in the introduction).
\medskip

To be more precise, in Figure \ref{Figure 10}, $\widetilde{f}$ is the profile curve corresponding to $\widetilde{M}$. We consider two circles of radius $d$, one of which is denoted in Figure \ref{Figure 10} by $C_d$, with centers at the points $(p\pm d \pm\frac{r}{2}, c+d+\frac{r}{100})$. Then, $\widetilde{f}$ is formed by  modifying $f$ within $[a,b]$ about $p \in (a,b)$ so that $\widetilde{f}$ does not intersect the circles $C_d$. We may form $\widetilde{f}$ so that $\widetilde{f}\geq f$, $\widetilde{f}(p) = f(p)+L$, $\widetilde{M}, \widetilde{f}$ satisfy condition (5) of Theorem \ref{third theorem} outside the ball $B(f(x_0), 2d)$, and $\widetilde{f}$ has only one critical point at $p$ inside the interval $[p-\epsilon, p+\epsilon]$. We may also choose $\widetilde{f}$ so that $|\widetilde{f}'|$ has a lower bound, independent of $\delta, d, r$ in the intervals $[p-\epsilon/2, p-d-\frac{r}{2}]$ and $[p+d+\frac{r}{2}, p+\frac{\epsilon}{2}]$. Note that the $\frac{r}{100}$ term is included in the definition of $C_d$ to ensure that we may satisfy the critical point condition (5) in the interval $[p-\epsilon, p+\epsilon]$. Also, based on our choice of $c$ and $d$, $\widetilde{M}$ may not be mean convex near $C_d$. 
\medskip

\begin{figure}
\centering
\includegraphics[scale = .45]{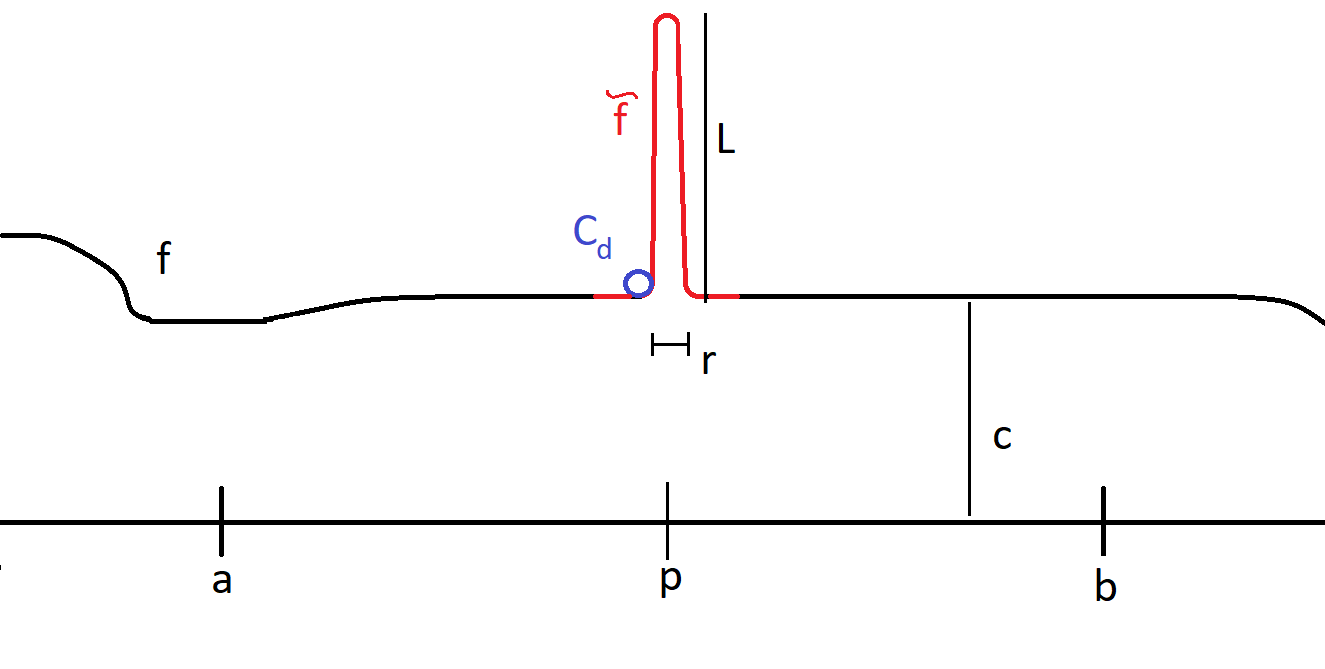}
\caption{The setup for the barriers.}
\label{Figure 10}
\end{figure}

As in the proof of Theorem \ref{second theorem}, we will first use barriers to show that the conclusion of Theorem \ref{third theorem} holds for $C^0$ norm, i.e.\ that $\widetilde{M}$ will flow to be arbitrarily close in $C^0$ to $M$, depending on sufficiently small choices of parameters $\delta, d, r$. 
\medskip

Before we choose which parameters to use in the construction of $\widetilde{f}$, we start by rescaling to make the eventual application of the Brakke regularity theorem and Lemma \ref{refinement} clearer. The content of the following lemma is that we may rescale $M$ at every point of the $(\delta,c)$-cylindrical region so that it is as close as we want to flat in as large a neighborhood as we want, with a rescaling factor depending on $c$. 
\medskip

From here on, let $M^\sigma$ be the surface $M$ rescaled by $\sigma$ around $p$, and let $\Omega^{\sigma}$ be the $\epsilon$-tubular neighborhood of the hypersurface given by the rotation of
$$\{(p+\sigma t, \sigma f(p)\,|\,t \in (-\epsilon, \epsilon)\}\subset \mathbb{R}^2$$
around the axis containing $[p-\sigma \epsilon,p+\sigma \epsilon]$. For both $M^{\sigma}$ and $\Omega^{\sigma}$ translate these sets, taking $p$ to the origin. Denote by $M_t^{\sigma}$ the parabolically rescaled mean curvature flow where the $t$ represents the rescaled time parameter.

\begin{lem}\label{scale factor lemma} With the notation set in the previous paragraph, for every $R, \overline{\epsilon} > 0$, we may find $\delta, \epsilon \ll 1$ and $\sigma \gg 0$ depending on $c, R, \bar{\epsilon}$ such that for every $q \in \Omega^{\sigma}\cap B(0,R/2)$, 

$M_t^{\sigma}\cap B(q,R/2)$ is $\overline{\epsilon}$-close in $C^2$ to $T_0 M_t^{\sigma}\cap B(0,R)$ for $t \in [0,100]$ and $\Theta(M_{t}^{\sigma}, x, r^*) \leq 1+\frac{\overline{\epsilon}}{2}$ for $r^* \in [1,2]$ and $x \in B(q,R/2)$.
\end{lem}

Note that if a surface $N$ is a graph over $M^{\sigma}$ that is $\delta'$-close in $C^2$ norm to $M^{\sigma}$, then $N^\frac{1}{\sigma}$ will be a graph over $M$ that is $\frac{\delta'}{\sigma}$-close in $C^2$ norm to $M$. The idea is that we will use the rescaling in combination with Lemma \ref{refinement} to find that $\widetilde{M}_t$ will be $\delta_0$-close in $C^2$ in $B(p,\epsilon)$ to $M$, assuming that it is $C^0$ close. As in Theorem \ref{second theorem}, a pseudolocality argument for $\widetilde{M}_t$ outside $[p-\epsilon, p+\epsilon]$ ensures that $\widetilde{M}_t$ will be $\delta_0$-close to $M$ in $C^2$. From this, we will conclude that $\widetilde{M}_t$ flows to a round point. In order to do this, we proceed as follows: 
\medskip

\begin{enumerate} 
\item Let $\delta_0>0$ such that if a smooth surface $N$ is $\delta_0$-close in $C^2$ to $M$, then $N$ flows smoothly to a round point. Pick $C'$ such that if $N$ satisfies $|A|<C'$ in $B(0,R)$, then $N$ is $\frac{\delta_0}{2}$-close in $C^2$ to the plane $T_0 N$ in $B(0,R)$. In Lemma \ref{refinement}, apply the choice of $C'$ to find an $\epsilon_0>0$ for which the conclusion of Lemma \ref{refinement} holds. 
\item For $R = 1000$ and $\overline{\epsilon} = \epsilon_0$, obtain a scale factor $\sigma$ in Lemma \ref{scale factor lemma}, such that $\sigma > \frac{10^4}{\delta_0}$.  With this choice of $\sigma$, for every $q \in \Omega^\sigma$ and $x \in B(q,R)$, $\Theta(M_{t}^{\sigma}, x, r^*) < 1 +  \frac{\overline{\epsilon}}{2}$ for $r^* \in [1, 2]$ and $t \in [0,100]$. 
\end{enumerate}

The point of choosing $\sigma > \frac{10^4}{\delta_0}$ is to ensure that the difference in $C^2$ between a plane and the $(\delta, c)$-cylindrical $M^{\sigma}$ is sufficiently small in $B(0,1000)$ for our purposes, taking $\delta \ll c/\sigma$. 
\medskip 

The goal now is to show that $\widetilde{M}_t$ will by some time be close in $C^0$ to $M$ depending on a small enough choice of parameters $\delta, d, r, \epsilon$. Then, we will use the structure of $\widetilde{M}_t$ to prove that the Gaussian densities will be small enough to apply Lemma \ref{refinement} and conclude the theorem.
\medskip

As in Lemma \ref{nearly graphical lemma}, we show that $\widetilde{M}^\sigma_t$ must become arbitrarily close in $C^0$ norm to $M^\sigma$ by some small time $T_0 < 50$, by choosing parameters $\delta, d, r$ small enough. This uses the recently-constructed ``ancient pancakes'' of Wang \cite{Wa11} and Bourni-Langford-Tinaglia \cite{BLT1} as barriers. 
\medskip

By the ``spike,'' we mean the glued-in region $\widetilde{M}^{\sigma}\setminus M^{\sigma}$ and the subsequent flow of this region. In particular, by the ``spike,'' we mean the graph $\widetilde{f}^{\sigma}_t$ restricted to the interval $[p-\sigma r/2, p+\sigma r/2]$. The idea is to control $\widetilde{M}^{\sigma}_t$ near $C_d$ and use an ancient pancake as a barrier for the flow of the spike, in order to show that $\widetilde{M}^{\sigma}_t$ is $C^0$ close to $M^{\sigma}$ after some short time. This barrier argument is quite similar to that of Theorem \ref{second theorem}. The only difference is that this requires using a family of ancient pancakes of width depending on time, with each pancake acting as a barrier for a short amount of time. 
\medskip

We first note that it is enough to control the position of the maximum of $\widetilde{f}_t^{\sigma}$, i.e. the ``tip'' of the spike. 
\medskip

By the Sturmian theory of \cite{Altschuler1995}, we have that the number of local minima and maxima of $\widetilde{f}^\sigma_t$, the flow of $\widetilde{f}^\sigma$, is nonincreasing and by Angenent's general Sturmian theory~\cite{Ang88}, this number drops exactly at the double zeros. This means that any inflection points disappear instantaneously and the number of critical points is nonincreasing and drops when two critical points come together. Since the only critical point of $\widetilde{f}^\sigma$ in the interval $[p-\epsilon\sigma, p+\epsilon\sigma]$ is at $p$ and since there is a uniform lower bound on $|\widetilde{f}'|$, independent of $\delta, d, r$ in the intervals $[p-\epsilon/2, p-d-\frac{r}{2}]$ and $[p+d+\frac{r}{2}, p+\frac{\epsilon}{2}]$, there will remain a single positive local maximum of $\widetilde{f}^{\sigma}_t$ in the interval $[p-\frac{\epsilon\sigma}{2}, p+\frac{\epsilon\sigma}{2}]$ for a uniform amount of time independent of $\delta, d,r$. 
\medskip

The barriers we use will intersect the flow near $C_d$. In order to apply barriers around the spike, we must control the behavior of the flow where the barriers will intersect. So, we will find control on the flow near $C_d$. Note that if $c$ is large compared to $d$, the curve shortening flow of the circles $C_d$ will nearly correspond, for short time, to the motion of the profile curve of a rotationally symmetric $S^{n-1} \times S^1$ with the same profile curve(s). So, up to fixed error for small time, we have that the curve shortening flow of $C_d$ serve as barriers for the motion of the profile curves $f_t$. This allows us to define the ``width'' $r_t$ of the spike as $|p_1 - p_2|$, where $p_1, p_2$ are the unique points in the interval $[p-d-r/2, p+d+r/2]$ such that $f(p_1), f(p_2) = c+d/2$. The fact that these points are unique comes from the Sturmian theory fact that there is one critical point in this interval. By the reasoning with $C_d$, $r_t$ can be bounded above by $r +d - \sqrt{d^2 - 2t}$, for a uniform amount of time depending only on $d$. 
\medskip
 
 Recall now that an ancient pancake is asymptotically the surface of rotation given by a grim reaper curve~\cite{BLT1}. As the ancient pancake is taken to be thinner, it will smoothly approximate the surface of rotation of a grim reaper.
\medskip
 
We will use the asymptotics of the ancient pancake in order to understand their behavior as barriers, and then to find $C^0$ control on $\widetilde{M}^{\sigma}_t$ as stated above. Note that a grim reaper of width
$r  + d - \sqrt{d^2 - 2t}$ translates at a speed approximately $\frac{1}{r + d -  \sqrt{d^2 - 2t}}$. Our family of barriers will approximate grim reapers with widths $r_t$, and this family will have infinite displacement as $r \to 0$. Indeed,  
$$\int_0^{d^2/2} \frac{1}{r + d - \sqrt{d^2 - 2t}}\,dt = [(d+r)\log{(\sqrt{d^2 - 2t} + d + r)} + \sqrt{d^2 - 2t}]\mid_0^{d^2/2}$$
 which diverges as $r \to 0$. This will approximate the motion of the tips of an ancient pancake as the width is taken to zero. 
 \medskip
 
Note that each ancient pancake is backwardly asymptotic to a slab: two parallel non-intersecting planes. And to be more specific, the distance between these two planes is what we call the width of the ancient pancake.  
 \medskip

  With this in mind, we consider a family of flows $\{P_i\}_{i=1}^n$ consisting of ancient pancakes defined on the uniform partition of $[0,d^2/2]$ by $n$ equal length intervals. Here, each $P_i(t)$ is an ancient pancake defined on times $t \in [(i-1)d^2/2n, i d^2/2n]$ and its width is $r  + d - \sqrt{(d^2 - 2(id^2/2n))}$ initially with tips positioned so that the tip of $P_i(id^2/2n)$ coincides with the tip of $P_{i+1}(id^2/2n)$. For each $w \ll 1$, there is $h_w>0$ so that in the profile plane an ancient pancake $P_i$ of width $w_1 < w$ will be arbitrarily close to the width $w_1$ slab in the complement of the $h_w$ neighborhood of the tip of $P_i$. Also, we may choose $h_w \to 0$ as $w \to 0$. We will ultimately let $w$ equal $d + r$ which will in turn be taken to be small, and each $P_i$ will have width bounded by $w$.    
\medskip

From now on, we will consider $P_i$ to be rescaled by $\sigma$. All parameters associated to width and time for ancient pancakes are written before rescaling by $\sigma$, for compactness of notation. This allows us to position and rescale the pancake $P_1$ so that $P_1$ encapsulates the spike, choosing $r$ small enough---that is, we choose a rescaled $P_1$ so that it only intersects $\widetilde{f}^{\sigma}$ where $\widetilde{f}^{\sigma} < \sigma (c+d/10)$ in the interval $[p-\sigma(d+r/2), p+\sigma(d+r/2)]$. Note that the estimate on the width $r_t$ of the spike controls where the pancake intersects $\widetilde{M}^{\sigma}_t$. Using $P_1$ as a barrier, this implies that the tip of the spike (that is, the maximum of the graph of $\widetilde{f}^{\sigma}_t$) will be contained in the flow $P_1(t)$ for the times $[0, d^2/2n]$ (before rescaling of the time interval), or until the height of $P_1$ is less than $h_w$. If the height of $P_1$ is less than $h_w$, then from our choice of $h_w$ below we will be done so we will suppose this does not occur. By our choices of $P_i$ along with the fact that wider pancakes of a given diameter contain narrower pancakes (i.e. contained in a smaller slab) of the same diameter, we have that the spike at time $(i-1)d^2/2n$ will be contained in the region bounded by $P_i((i-1)d^2/2n)$ initially and will remain so for the time it is defined.
\medskip
 
Let $h$ be the initial height of $P_1$, which can be taken to be the height of the spike initially, i.e.\ the maximum of $\tilde{f}_0^{\sigma}$ over $[p-\sigma(d/2+r/2), p+ \sigma(d/2 + r/2)]$. Since the tips of the pancakes are taken to coincide, the tips of each of the $P_i((i-1)d^2/2n)$ have height $h_i$ bounded above and below as in (\ref{htbound}). This uses the fact that asymptotically the profile curve of the pancakes agrees with a grim reaper~\cite{BLT1}:
 \begin{equation}\label{htbound}
 \sigma (c-\delta)+ h_w < h_i < h - \sum_{k=1}^{i -1} \frac{1}{r + d - \sqrt{d^2 - 2((k-1)d^2/2n)}}\frac{\sigma d^2}{2n} + O(r,d),
  \end{equation}
where $O(r,d)$ is the error measuring the deviation of the asymptotics of the ancient pancakes from the rotated grim reaper. That is, $O(r,d)\to 0$ as $r,d \to 0$, i.e.\ the error from the rotated grim reaper improves as the widths of the ancient pancakes are taken to zero with fixed diameters. The bound from below in (\ref{htbound}) follows trivially, from the definition of $h_w$ and our barriers, since $\tilde{f}_t^\sigma \geq f_t^{\sigma}$, where $f_t^{\sigma}$ is $\sigma \delta$-close to $\sigma c$. The sum in (\ref{htbound}) is just a lower Riemann sum for the displacement integral for the grim reaper above. So, if $d, r, \delta$ are taken to be small enough, the quantity $h_i$ can be taken to be arbitrarily close to $\sigma c$ for large enough $i$. That is, we may arrange so that $h_i$ is close to $\sigma c$, no matter what $h$ is initially, since $h_w \to 0$ as $w \to 0$.
\medskip

Since the family of ancient pancakes $\{P_i\}$ serves as barriers, the maximum of $\widetilde{f}^{\sigma}_t$ at $p$ will be bounded by $h_i$. So, this maximum can be arranged to be arbitrarily close to $\sigma c$ no matter its initial height by picking $d$, $r$ small enough again with $w = d+r$. That is, for each $\delta^*>0$, there exists a choice of $\delta, d, r>0$ and $T^*$ such that for $\widetilde{M}^{\sigma}$ corresponding to $\delta, d, r>0$, $\widetilde{M}^{\sigma}_{T^*}$ is $\delta^*$-close in $C^0$ to a plane in $B(p, \sigma \frac{\epsilon}{2})$. 
\medskip

In order to upgrade the $C^0$ estimates to $C^2$ estimates, we will control the Gaussian densities of $\widetilde{M}^{\sigma}_t$. This will use the particular structure of the flow of $\widetilde{f}^\sigma_t$, which allows for applications of Sturmian theory, inside the interval $[p-\frac{\epsilon}{2}\sigma, p+\frac{\epsilon}{2}\sigma]$ (note in practice $\epsilon \sigma \gg 1$). An application of Sturmian theory~\cite{Altschuler1995} says that $\widetilde{f}_t$ has only one local maximum in $[p-\frac{\epsilon}{2}\sigma, p+\frac{\epsilon}{2}\sigma]$. Moreover, since a local maximum of a graph has positive geodesic curvature with respect to the downward pointing normal, the mean curvature of $\widetilde{f}^\sigma_t$ at $p$ will remain positive for a time independent of $d,r$. Note also that $\widetilde{f}^{\sigma}_t \geq f^{\sigma}_t$ for each $t$. This implies that once the flow is $\delta^*$-close in $C^0$ to a plane as a graph over $B(p, \sigma \frac{\epsilon}{2})$ by time $T_0$, it will remain so for a time independent of $d,r$ because the local maximum of $\widetilde{f}^{\sigma}_t$ at $p$ will monotonically approach the axis of rotation. By pseudolocality, for each $\delta^*$, there exists a choice of $\delta, d, r$ and $T_0$ such that $\widetilde{f}^{\sigma}_t$ will be $\delta^*$-close in $C^0$ to the radius-$\sigma c$ cylinder over $[p-\frac{\epsilon}{2}\sigma, p+ \frac{\epsilon}{2}\sigma]$ for $t \in [T_0, \widetilde{T}]$, where $\widetilde{T}$ is independent of $\delta, d, r$, as in the proof of Theorem \ref{second theorem}. By possibly rescaling with $\sigma$ more, we may ensure that $\widetilde{T}>100$. Also, note that $T^* \to 0$ as $\delta, d, r \to 0$, just as in Lemma \ref{nearly graphical lemma}.
\medskip

 Our arrangement of $\widetilde{M}^\sigma$ is summed up in the following lemma. See Figure \ref{Figure 12}, which corresponds to Lemma 5.2.
\medskip

\begin{lem}\label{lemma lemma} Let $q \in \Omega^\sigma$. Then for every $\epsilon^*, \delta^* \ll 1$ and $T_0 < 50$ there is a choice of $\delta, d, r$ giving $\widetilde{M}^{\sigma}$ and a hyperplane $P$ and codimension 2 plane $\mathcal{L} \subset P$ such that the following holds:
\begin{enumerate}
\item $\widetilde{M}^\sigma_{t} \cap B(q,1000)$ is $\epsilon^*$-close in $C^2$ norm to $P$ in the complement of $T_{2d} \mathcal{L}$, the radius-$2d$ tubular neighborhood  of $\mathcal{L}$,
\item In $\mathcal{T}:= T_{2d}\mathcal{L} \cap B(q,1000)$, $\widetilde{M}^\sigma_{T_0}$ is a graph over $P$ of height bounded by $\delta^*$ for $t \in [T_0, 100]$. 
\end{enumerate} 
\end{lem}

\begin{figure}
\centering
\includegraphics[scale = .45]{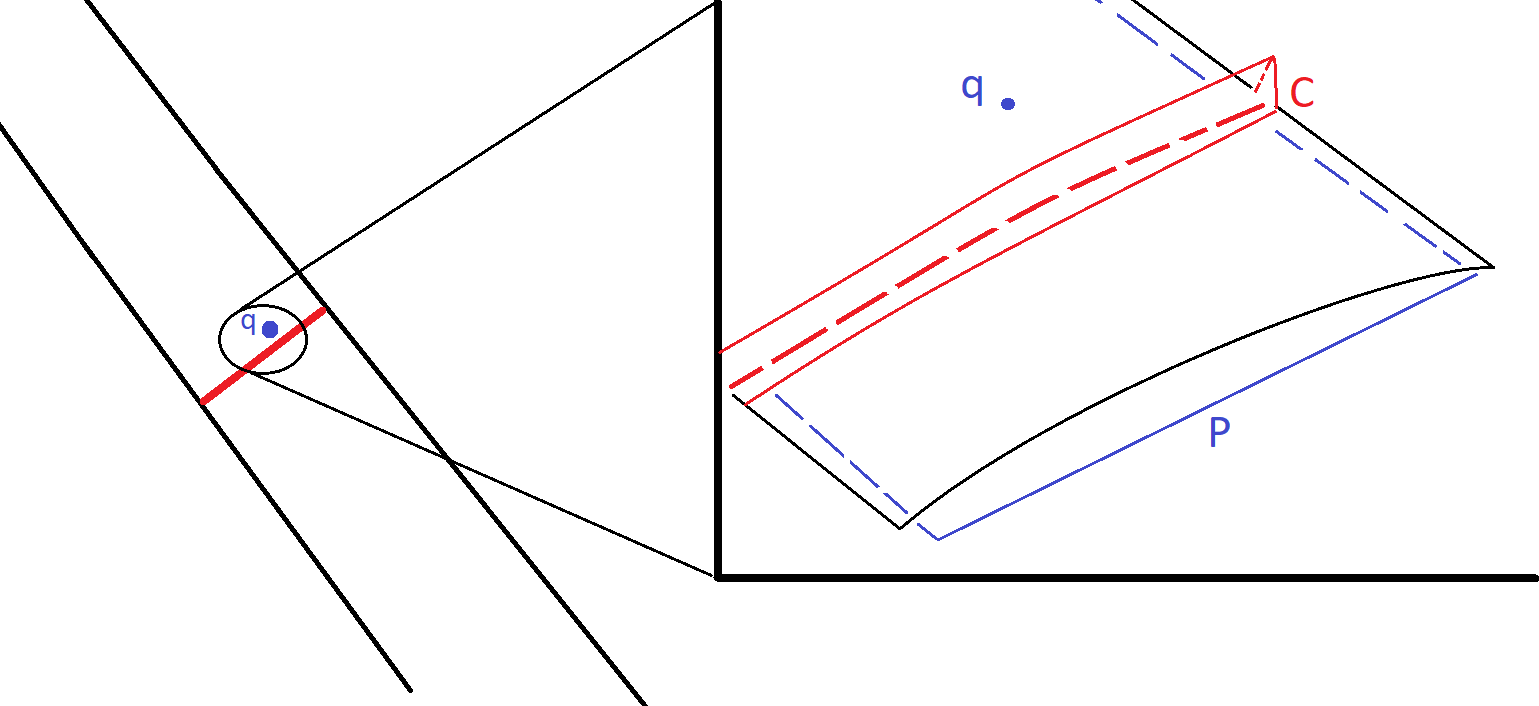}
\caption{A depiction of the flow after a short time.}
\label{Figure 12}
\end{figure}
\medskip

By the fact that $\widetilde{f}^{\sigma}_t$ will have only one critical point at $p$ for a time independent of $\delta, d,r$ and by an application of Lemma \ref{lemma lemma}, we find that the $C^1$ norm of $\widetilde{M}^{\sigma}_t$ at $t = T_0$ in the complement of $T_{2d}\mathcal{L}\cap B(0,1000)$ goes to zero as $\delta, d,r \to 0$. Now, we note that there are uniform entropy bounds (giving uniform area ratio bounds) for $\widetilde{M}^{\sigma}_t$ independent of $\delta, d, r$. Choosing $r \ll d$ small enough and applying the uniform area ratio bounds, we find $\delta, d, r$ small enough so that $\Theta(\widetilde{M}^{\sigma}_t \cap T_{2d}\mathcal{L}\cap B(0,1000), x, r^*)\leq \frac{\bar{\epsilon}}{4}$ for $x \in \widetilde{M}^{\sigma}_t \cap T_{2d}\mathcal{L}\cap B(0,1000)$ and $r^* \in [1,2]$. Then, applying Lemma \ref{lemma lemma} and the $C^1$ bounds on $\widetilde{M}^{\sigma}_t$ in the complement of $T_{2d}\mathcal{L}\cap B(0,1000)$, we find $\delta, d, r$ small enough such that that $\Theta(\widetilde{M}^{\sigma}_t, q, r^*) \leq 1 + \overline{\epsilon}$ for $x \in \widetilde{M}^{\sigma}_t \cap B(0,1000)$ and $r^* \in [1,2]$ at $t = T_0$. Huisken monotonicity then implies that $\Theta(\widetilde{M}^{\sigma}_t, x, r^*) \leq 1 + \overline{\epsilon}$ for $t \in [T_0 + 25, T_0 + 50]$ and $r^* \in (0, 2]$. Applying Lemma \ref{refinement} and our choices of constants, we find that $\widetilde{M}^{\sigma}_t$ is $\delta_0/2$-close in $C^2$ to a plane in $B(0,1000)$. By our choice of $\sigma, \delta$, $\widetilde{M}^{\sigma}_t$ is $\delta_0$-close in $C^2$ to $M^{\sigma}$ in $B(0,1000)$ for $t \in [T_0 + 25, T_0 + 50]$. Applying pseudolocality outside $B(0,1000)$, as in Theorem \ref{second theorem}, we have that $\widetilde{M}^{\sigma}_t$ is $\delta_0$-close in $C^2$ to $M^{\sigma}$, meaning that $\widetilde{M}_t$ flows to a round point. This completes the proof of Theorem \ref{third theorem}.

\section{Applications}

In this section, we will discuss applications of Theorems \ref{second theorem} and \ref{third theorem}. First, we will prove Corollary \ref{no GH limit corollary}, which is summarized by Figure \ref{spiky ball figure}.

\begin{proof}[Proof of Corollary \ref{no GH limit corollary}]
This corollary is a consequence of iterated applications of Theorem \ref{second theorem}. By a ``spike,'' we mean the glued-in subset of $\widetilde{M}$ found in Theorem \ref{second theorem} which is a perturbation of a tubular neighborhood of a curve. Let $M^0$ be the round unit sphere centered at the origin. Then, if we have defined $M^i$, we define $M^{i+1}$ by attaching an outward-pointing length one ``spike'' using Theorem \ref{second theorem} with a base point $p_{i+1} \in M^i \cap M^0$ and $L=1$. We take the parameters small enough in each application of Theorem \ref{second theorem} so that $M^i \cap M^0$ is nonempty for each $i$. So, by construction, each $M^i$ shrinks to a round point. Since each $M^i$ contains a round unit sphere, there is a uniform lower bound on the existence time for $M^i_t$ independent of $i$. However, as noted in the survey ~\cite{SormaniHowManifoldsConverge12},  balls of radius $\frac{1}{2}$ located at the tip of each spike are all disjoint. Since there are infinitely many of them in the limit, there is no Gromov-Hausdorff limit of $M^i$.
\end{proof}
Now, we will prove Corollary \ref{no GH limit with area unbounded}, which is a corollary of Theorem \ref{third theorem}.
\begin{proof}[Proof of Corollary \ref{no GH limit with area unbounded}]
This corollary is summarized in Figure \ref{entropy example}. Using Theorem \ref{second theorem}, let $M^0$ be a round unit sphere with a rotationally symmetric spike of length $1$ attached, such that $M^0$ shrinks to a round point. We may choose the parameters of the spike from Theorem \ref{second theorem} (namely $N$) small enough so that the spike is close enough to cylindrical to apply Theorem \ref{third theorem}. By a ``pancake,'' we mean the glued-in subset of $\widetilde{M}$ found in Theorem \ref{third theorem} which is a perturbation of a tubular neighborhood of a curve rotated around an axis. We use Theorem \ref{third theorem} with $L=1$ applied to a point halfway up the spike to find $M^1$. The surface $M^1$ will look like $M^0$ with a thin pancake attached halfway up the spike. Then, to construct $M^2$, we attach another rotationally symmetric spike of length $\frac{1}{10}$ at the tip of the spike on $M^1$ using Theorem \ref{second theorem}. Again, we apply Theorem \ref{third theorem} with $L=1$ at a point halfway up this new spike, having chosen the parameters in Theorem \ref{second theorem} small enough for the new spike so that we may apply Theorem \ref{third theorem}. Inductively, if $M^i$ is constructed, then $M^{i+1}$ is constructed by attaching a rotationally symmetric spike of length $\frac{1}{10^{i}}$ at the tip of the spike constructed for $M^i$ and we apply Theorem \ref{third theorem} with $L=1$ at a point halfway up this new spike.
$\medskip$

By construction, each $M^i$ flows to a round point and must exist for a uniform time since each $M^i$ contains the round unit sphere. Since the series $\frac{1}{10^i}$ converges, we have that $\diam(M^i)$ is uniformly bounded. We are attaching thin pancakes of radius $1$ to each spike, so $\Area(M^{i+1}) \geq \Area(M^i)+2\pi$ and $\Area(M^i) \to \infty$.  
$\medskip$

Now, by the construction in Theorem \ref{third theorem}, for each $i$, the flow of $M^{i+1}$ will be close in $C^2$ to $M^i$ after some small time $T_i$. That is, $M^{i+1}$ will flow for time $T_i$ so that $M^{i+1}_{T_i}$ is close enough in $C^2$ to $M^i$ so that $M^{i+1}\in \overline{\Sigma}$. We may easily arrange so that $\sum_i T_i < \infty$ and the tail of this sequence goes to zero. For each $j$, each $M^i_t$ will be uniformly $C^2$ close to the flow of $M^j_t$ by a small positive time independent of $i$. Thus, we have that $\Area(M^i_t) < C(t)$ for a $C(t)$ independent of $i$. 
\end{proof}

\begin{proof}[Proof of Corollary \ref{highentropy}] This corollary is closely related to Corollary \ref{no GH limit with area unbounded}, and it is summarized by Figure \ref{entropy example}. Fix $\delta, E > 0$. The example for this corollary will be constructed via a finite-step inductive procedure similar to what is done above. Form $M^0$ by attaching a spike of length $\frac{\delta}{10}$ to the round unit sphere. Construct $M^{i+1}$ by attaching a spike of length $\frac{\delta}{10^{i+2}}$ to the tip of the spike attached to $M^{i}$. For each $M^{i}$, glue in pancakes on the spike using Theorem \ref{third theorem} as above so that each pancake has radius $L = \frac{\delta}{10}$. By construction, each $M^i$ will flow to a round point and are all within $\delta$ of the round unit sphere in the Hausdorff distance. Since $\Area(M^{i+1}) \geq \Area(M^i)+\frac{2\pi \delta^2}{100}$, we have that $M^{i}$ has arbitrarily large area for $i$ large. By considering an $F$ functional (see \cite{ColdingMinicozziGeneric}) at the scale of the pancakes centered near the middle of the attached spike of $M^{i}$, we have that the entropy can be taken to be arbitrarily large for $i \gg 1$, in particular larger than $E$, since the area is arbitrarily large in a small neighborhood.
\end{proof}

\begin{proof}[Proof of Corollary \ref{space filling surface sequence}]
Fix $M^0$ to be a unit sphere. Find a maximal $\frac{1}{10}$-separated net $\mathcal{N}^0$ on $M^0$. For each point in $\mathcal{N}^0$, apply Theorem \ref{second theorem} with $L=2$ and $C=0$ to construct an inward-pointing spike. This must be done one at a time, and the width of each spike attached will vary. Let $M^1$ be the surface with all spikes attached to $\mathcal{N}^0$. Let $C_1$ bound the second fundamental form of $M^1$, i.e.\ $|A|^2 \leq C_1$. Note that each application of Theorem \ref{second theorem} gives an $r$ which controls the curvature of the spike added, so $C_1$ is controlled by the reciprocal of the smallest $r$ used in the application of Theorem \ref{second theorem} to the points $\mathcal{N}^0$.
Now, we pick two maximal $\frac{1}{10^2 C_1}$-separated nets, $\mathcal{N}^1_1$ and $\mathcal{N}^1_2$, on $M^1$, and we attach spikes via Theorem \ref{second theorem} with $L=2$ and $C=0$ at the points $\mathcal{N}^1_1$ and $\mathcal{N}^1_2$. We choose the spikes attached at the points $\mathcal{N}^1_1$ and $\mathcal{N}^1_2$ to be pointing in opposite directions, so that the spikes attached at $\mathcal{N}^1_1$ are inward-pointing and the spikes attached at $\mathcal{N}^1_2$ are outward-pointing. Then, we let $M^2$ be the surface obtained by this process. Now, we will construct $M^{i+1}$ iteratively as above. If we have $M^i$, then we find $C_i$ bounding the second fundamental form of $M^i$. We pick two maximal $\frac{1}{10^{i+1}C_i}$-separated nets, $\mathcal{N}^i_1$ and $\mathcal{N}^i_2$, on $M^i$, and we attach spikes via Theorem \ref{second theorem} with $L=2$ and $C=0$ at the points $\mathcal{N}^i_1$ and $\mathcal{N}^i_2$. We choose the spikes to be inward and outward-pointing for $\mathcal{N}^i_1$ and $\mathcal{N}^i_2$ as in the base case. We let $M^{i+1}$ be the surface obtained from adding all the spikes to $M^i$ as specified. By construction, each $M^i$ shrinks to a round point by Theorem \ref{second theorem}. By the same reasoning as in Corollary \ref{no GH limit with area unbounded}, we may conclude that $\Area(M^i_t) < C(t)$ for $C(t)$ independent of $i$. 
$\medskip$

Now we will prove that in the Hausdorff distance, $M^i$ converges to a set $\mathcal{K}$ that contains a unit ball $\mathcal{B}$, so $M^i$ is space-filling in the limit. We will do this by contradiction.
$\medskip$

Let $\mathcal{B}$ be the unit ball whose boundary is $M^0$. Suppose there exists a point $x \in \mathcal{B}$ such that for some $c>0$, $M^i \cap B(x,c) = \emptyset$ for all $i$. This implies that for all $i$ large enough, the normal lines\footnote{Here, we mean the oriented normal lines which are inward-pointing for $\mathcal{N}^i_1$ and outward-pointing for $\mathcal{N}^i_2$.} to $M^i$ at $\mathcal{N}^i_1$ and $\mathcal{N}^i_2$ do not intersect $B(x,c)$. This is because if a normal line to $M^i$ at a point in $\mathcal{N}^i_1$ or $\mathcal{N}^i_2$ did intersect $B(x,c)$, then one would construct a spike of length $L=2$ around that normal line (in order to construct $M^{i+1}$) which would contradict the assumption that $M^i\cap B(x,c)=\emptyset$ for all $i$. Now, for each $i$ large, let $p_i\in M^i$ be the point that minimizes the distance $d(x,M^i)$. 
Since $p_i$ minimizes the distance from $x$ to $M^i$, the normal line to $M^i$ at $p_i$ passes through $x$. Suppose without loss of generality that $p_i$ is on the inward-pointing side of $M^i$, meaning that the normal line to $p_i$ oriented inward intersects $x$. The rest of the argument works just as well for outward-pointing by just replacing $\mathcal{N}^i_1$ with $\mathcal{N}^i_2$. Now, let $n^i_1\in \mathcal{N}^i_1$ be a point that minimizes the distance $d^{M^i}(p_i, \mathcal{N}^i_1)$, where $d^{M^i}$ denotes the intrinsic distance in $M^i$. Since $\mathcal{N}^i_1$ is $\frac{1}{10^{i+1}C_i}$-separated, this means that $d^{M^i}(p_i, n^i_1) \leq \frac{1}{10^{i+1}C_i}$. Since for large $i$ the distance between $p_i$ and $n^i_1$ is much smaller than the scale of the curvature of $M^i$, the inward-oriented normal line to $M^i$ at $n^i_1$ will be arbitrarily close to the inward-oriented normal line to $p_i$ as $i$ gets large. Since the inward-oriented normal line to $p_i$ intersects $x$, this means that for $i$ large, we may find an inward-oriented normal line to $M^i$ at some point of $\mathcal{N}^i_1$ that intersects $B(x,c)$. This contradicts the fact that the normal lines to $M^i$ at $\mathcal{N}^i_1$ and $\mathcal{N}^i_2$ do not intersect $B(x,c)$. Thus, there is no such $x$ and $c$, and the sequence $M^i$ becomes dense in $\mathcal{B}$.
\end{proof}

\begin{proof}[Proof of Corollary \ref{space filling hypersurface}]
The construction of the sequence $M^i$ is the same as in Corollary \ref{space filling surface sequence}. By construction of the spikes as in Theorem \ref{second theorem}, all $M^i_t$ must be some perturbation of $M_t$ by some time. Clearly, we may choose all parameters such that $M^i_t$ is arbitrarily close to $M_t$ by any $t_0 \ll 1$. Then, using continuity of the flow under initial conditions, we obtain this corollary. 
\end{proof}

\section{Concluding Remarks}
The following question is natural in light of Theorem \ref{first theorem}:

\begin{Question}\label{question} Does each embedded interval in $\mathbb{R}^{n+1}$ have a tubular neighborhood whose $C^2$ perturbations shrink to a round point? 
\end{Question}
There is some reason to believe this statement could be true. By Huisken's theorem, any closed, convex hypersurface flows smoothly to a round point, and this includes includes arbitrarily long closed convex surfaces. Fixing a radius, a capped-off cylinder of any length must shrink to a point in a uniformly bounded amount of time (by comparison with the round cylinder). Turning to the nonconvex case, a sufficiently small tubular neighborhood of any curve in $\mathbb{R}^{n+1}$ has mean curvature comparable to that of the convex cylinder of the same radius, and we expect this flow to behave as in the convex case. With that said, it could be possible to answer Question \ref{question} by considering tubular neighborhoods not of constant width but of varying width tailored to the geometry of the underlying curve.
$\medskip$

It is natural to consider analogous versions of Theorem \ref{second theorem} and \ref{third theorem} for other geometric flows, such as the Ricci flow. A result of Bamler-Maximo~\cite{BamlerMaximo17} states that under certain curvature conditions, Ricci flows with almost maximal extinction times must be nearly round. This is similar in spirit to our work in the proof of Theorem \ref{first theorem} as we are interested in controlling nearly convex flows up to their extinction time, which is itself near to the maximal extinction time of a comparable cylinder. However, our examples constructed from Theorem \ref{second theorem} will have extinction time very far from the maximal extinction time, interpreted in any reasonable sense. There do not appear to be other results in Ricci flow similar to the results in this paper. 
$\medskip$

For mean curvature flow in a curved ambient manifold or in higher codimension, analogous statements to Theorems \ref{second theorem} and \ref{third theorem} likely hold true, but their proof might not be simple. Our results depended on the existence of particular translating solitons and ancient solutions which serve as barriers, and no suitable analogues are known to exist for an arbitrary ambient manifold, although one would expect there to be ``approximate'' barriers in a small neighborhood of any point by scaling about it. In a similar vein, it is also unclear how to proceed in the higher codimension setting, due to the failing of the avoidance principle. It is possible one could reduce the higher codimension case to the codimension one case in this paper by working in a region where a $k$-dimensional flow approximately lies in a $k+1$ dimensional subspace of $\mathbb{R}^{n}$.

\begin{appendix}
\section{$C^2$ Stability and Nonconvex Surfaces Which Shrink to Round Points via Compactness-Contradiction}
In the appendix, we discuss how perturbations of surfaces which flow to round points will also flow to round points. The first result gives us that the set of smooth hypersurfaces which shrink to a round point is open under $C^2$ perturbations. Then, we prove a statement similar to Theorem \ref{first theorem} using a standard compactness argument.

\begin{thm}\label{appendix theorem 1} Let $\overline{\Sigma}$ be the set of smooth hypersurfaces whose mean curvature flow shrinks to a round point. For each $M \in \overline{\Sigma}$, there exists $\delta>0$ such that if $M^*$ is $\delta$-close in $C^2$, then the mean curvature flow $M^*_t$ exists smoothly until it flows to a round point.
  \end{thm}
\begin{proof}
Let $M \in \overline{\Sigma}$. Suppose that this statement is false. Then, there exists no such $\delta>0$ and there exists a sequence $M^i$ converging in $C^2$ to $M$ such that $M^i \notin \overline{\Sigma}$. Note that $M^i$ has uniformly bounded curvature over $i$. Let $M^i_t$ and $M_t$ be the mean curvature flow with initial condition $M^i_0 = M^i$ and $M_0 = M$. Since $M^i$ converges in $C^2$ to $M$, there is a uniform bound on $|A|^2$ for $M^i$, independent of $i$. Thus, there exists a uniform doubling time $T$ such that each mean curvature flow $M^i_t$ and $M_t$ exists smoothly for $t \in [0,T]$. Moreover, by definition of the doubling time, for each $i$ and $t \in [0,T]$, $M^i_t$ satisfies $|A|^2 < C$ for some constant $C$ independent of $i$ and $t$. Let $A^i_t$ be the second fundamental form of $M^i_t$. Then, by standard derivative estimates for mean curvature flow, there exists $C_{\ell}$ depending only on $C$, the dimension $n$, and $\ell$ such that
$$|\nabla^{\ell} A^i_t|^2 \leq C_{\ell} \big(1+\frac{1}{t^{\ell}}\big)$$
 for each $i$ and $t \in (0,T]$. By compactness for the second fundamental form~\cite[Theorem 5.2]{Cooper}, we find that for each positive integers $N, n$, there exists a subsequence of the sequence of flows $M^i_t$ which converges in $C^N$ to a smooth closed mean curvature flow $\widetilde{M}_t$ defined for $t \in [\frac{1}{n}, T]$. By a diagonal subsequence, there exists a subsequence of $M^i_t$ which converges in $C^{\infty}$ to a smooth mean curvature flow $\widetilde{M}_t$ defined for $t \in (0, T]$. We must show that $\widetilde{M}_t = M_t$ for $t \in [0,T]$.  By the uniform curvature estimate $|A|^2 < C$ for $M^i_t$ and the fact that $M^i$ converges to $M$ in $C^2$, we have that $\widetilde{M}_t$ converges to $M$ in $C^0$ as $t \to 0$. 
\medskip

Recall that the level set flow of $M$ is exactly $M_t$~\cite{EvansSpruck91}. Moreover, the level set flow $M_t$ is the maximal set-theoretic subsolutions with initial condition $M$~\cite{Ilmanen93}. Here, a set-theoretic subsolution is a flow of closed sets which satisfy the avoidance principle with respect to every other smooth compact mean curvature flow. If we define $\widetilde{M}_0 = M$, then $\widetilde{M}_t$ is a set-theoretic subsolution for $t \in [0,T]$ since $\widetilde{M}_t$ converges in $C^0$ to $M$ as $t \to 0$ and $\widetilde{M}_t$ satisfies the avoidance principle for $t \in (0,T)$. Thus, $\widetilde{M}_t \subseteq M_t$ for each $t \in [0,T]$. Since $M_t$ is connected and $\widetilde{M}_t$ is a smooth closed mean curvature flow, $\widetilde{M}_t = M_t$ for $t \in [0,T]$. 
\medskip
 
This implies that there is a subsequence of flows $M^i_t$ which converge in $C^{\infty}$ to $M_t$ for $t \in (0,T)$. 
\medskip

Now, we recall that mean curvature flow has continuous dependence in $C^{\infty}$ on initial conditions~\cite[Theorem 1.5.1]{Mantegazza}. If $M \in \overline{\Sigma}$, then as a consequence of the definition, there is a time $T^*$ such that $M_{T^*}$ is strictly convex. By continuous dependence in $C^{\infty}$, there exists $\delta^*>0$ such that if $N$ is $\delta^*$-close in $C^{\infty}$ to the initial condition $M$, then $N_{T^*}$ is strictly convex. By~\cite{Huisken84}, $N$ flows to a round point and so $N \in \overline{\Sigma}$. Applying this fact to $M_{T/2}$, using that $M_{T/2} \in \overline{\Sigma}$ and $M^i_t$ converges in $C^{\infty}$ to $M_t$ for $t \in (0,T)$, there exists $i^*$ such that $M^{i^*}_{T/2} \in \overline{\Sigma}$. This implies that $M^{i^*} \in \overline{\Sigma}$, and this is a contradiction of the assumption that each $M^i \notin \overline{\Sigma}$. This concludes the proof of the lemma.
\end{proof}

The next theorem will refine Theorem \ref{first theorem} and Theorem \ref{appendix theorem 1} in some regards. The following theorem shows particular dependence of the size of the perturbation on geometric characteristics of a convex surface.

\begin{thm}\label{appendix theorem} Let $\Sigma(d,C)$ be the set of closed embedded hypersurfaces $M^n \subset \mathbb{R}^{n+1}$ such that 
\begin{enumerate}
\item $\diam(M) <d$
 \item $|A|^2< C$
 \end{enumerate}
 
 Then there exists an $\epsilon(d,C) > 0$ such that if $M \in \Sigma(d,C)$ and $k_{min} > -\epsilon(d,C)$, then $M$ flows into a sphere under the mean curvature flow.
  \end{thm}

One can proceed following more or less as Petersen and Tao do in their note on the Ricci flow of nearly quarter-pinched manifolds~\cite{PetersenTao09}. Assume to the contrary that there is no such $\epsilon$. Take a sequence $\{M_n\}_{n = 1}^\infty \subset \Sigma(d,C)$ of hypersurfaces such that for each $n$, $k_{min} > -1/n$ yet none of the $M_n$ flow to spheres under the normalized mean curvature flow. Because the mean curvature flow is invariant under translation and by the diameter bound, all manifolds can be considered to be contained in $B_d(0)$ without loss of generality. 
$\medskip$

From our uniform curvature bounds, using that $B_d(0)$ is compact, we get that there is a cover $\{U_i\}$ of $B_d(0)$ so that each of the $M_n$ is given as a union of graphs with uniform $C^2$ bounds. Using Arzela-Ascoli we then attain a $C^{1,\beta}$-converging subsequence of graphs. Relabeling them, consider the sequence $M_n\to N$, a set in $\mathbb{R}^{n+1}$ locally given by $C^{1, \beta}$ graphs, so that $N$ is an immersed $C^{1, \beta}$ manifold. This is not a strong enough convergence to use our continuity on initial conditions to conclude directly that the flows converge though. 
$\medskip$

Our plan then is to use the flow to get uniform bounds not only on $|A|^2$ but also uniform bounds on all its derivatives. Then we could attain a smoothly convergent subsequence and so their flows converge. The problem is that we would want this new sequence to be ``offending" in that $k_{min} \to 0$ but all the $M_n$ do not flow to spheres; the second condition is clearly invariant under the flow but the first is not necessarily. First we need a time $t > 0$ when all the flows exist with bounded curvature in order to apply interior estimates. This is a simple consequence of the evolution equations:

\begin{lem}
There is a time $T > 0$ so that the flows of all $M_n$ through time $T$ have $|A_n(t)| < \overline{C}$, where $A_n$ denotes the 2nd fundamental form of the $n^{th}$ surface in the sequence above. 
\end{lem} 

Fixing $t_0 \in (0,T]$ we have that the sequence $M_n(t_0)$ is a collection of smooth manifolds, and from the usual interior estimates will have uniform bounds on $|\nabla^\ell A|$ in terms of the uniform bound $\overline{C}$, valid for $t > t_0$. We want to find out if the sequence is offending now. 
$\medskip$

The principal curvatures of a hypersurface $M$ are eigenvalues of its shape operator $S= \{h_i^j\}$, so we must study what happens to it under the flow. Inspired by \cite{PetersenTao09}, we adopt ideas from \cite{Hamilton86} concerning the proof of the tensor maximum principle of Hamilton therein.
$\medskip$

To do this, let $M^n \subset \mathbb{R}^{n+1}$ below stand for a compact hypersurface with shape operator $S$. The eigenvalues of $h_i^j$ are the principal curvatures, so $h_i^j$ is a positive semidefinite matrix if and only if $M$ is a convex hypersurface. The evolution equation of the shape operator $h_i^j$ under the flow is given by:
\begin{center}
$\frac{\partial h_i^j}{\partial t} = \Delta h_i^j + |A|^2 h_i^j$
\end{center}
Which we write in compacted notation as
\begin{center}
$\frac{\partial S}{\partial t} = \Delta S+ |A|^2 S$
\end{center}

In the following we closely follow \cite{Hamilton86} section 3; here for notational ease we only concern ourselves with matrices. We will work in just a coordinate patch of $M$, but this generalizes easily to tensors fields over manifolds, as in section 4 of \cite{Hamilton86}. 
$\medskip$

First, setting notation, let $X$ denote the set of positive definite matrices. We see that it is convex in each fiber. We define the tangent cone $T_f X$ at a point $f \in \partial X$ as the smallest closed convex cone with vertex at $f$ which contains $X$. It is the intersection of all the closed half-spaces containing $X$ with $f$ on the boundary of the half space. 
\begin{lem}  The solutions of $\frac{df}{dt} = |A|^2 f$ which start in the closed convex set $X$ will remain in $X$ if and only if $ |A|^2 f \in T_f X$ for all $f \in \partial X$. \end{lem}
\begin{proof}
We say that a linear function $\ell$ on $\mathbb{R}^{n^2}$ is a support function for $X$ at $f \in \partial X$ and write $\ell \in S_fX$ if $|\ell| = 1$ and $\ell(f) \geq \ell(k)$ for all other $k \in X$. Then $|A|^2 f \in T_f X$ if and only if $\ell(|A|^2 f) \leq 0$ for all $\ell \in S_f X$. Suppose $\ell(|A|^2 f)> 0$ for some $\ell \in S_f X$. Then
\begin{center}
$\frac{d}{dt} \ell(f) = \ell(\frac{df}{dt}) = \ell(|A|^2 f) > 0$
\end{center}
so $\ell(f)$ is increasing and $f$ cannot remain in $X$. To see the converse, note as in \cite{Hamilton86} that without loss of generality $X$ is compact. Let $s(f)$ be the distance from $f$ to $X$, with $s(f) = 0$ if $f \in X$. Then
\begin{center}
$s(f) =$ sup$\{\ell(f - k)\}$
\end{center}
where the sup is over all $k \in \partial X$ and all $\ell \in S_f X$. This defines a compact subset $Y$ of $\mathbb{R}^{n^2} \times \mathbb{R}^{n^2}$. By Lemma 3.5 in \cite{Hamilton86},
\begin{center}
$\frac{d}{dt}s(f) \leq$ sup$\{\ell(\phi(f))\}$ 
\end{center}
where the sup is over all pairs $(k, \ell)$ with $k \in \partial X$, $\ell \in S_k X$, and
\begin{center}
$s(f) = \ell(f - k)$
\end{center}
Note this can happen only when $k$ is the unique closest point in $X$ (using $X$ is closed and convex) and $\ell$ is the linear function of length 1 with gradient in the direction $f - k$. Now since $M$ is compact we assume that $|A|^2$ is bounded by a constant $C_M$ so we have that 
\begin{center}
$| |A|^2f - |A|^2 k| \leq C_M|f - k|$ 
\end{center}
Since $\ell(|A|^2k) \leq 0$ by hypothesis and $|f - k| = s(f)$ we have:
\begin{center} 
$\frac{d}{dt} s(f) \leq \frac{d}{dt} \ell(f - k) = \ell( |A|^2 f) \leq  \ell( |A|^2 f) - \ell(|A|^2 k) =  \ell( |A|^2 (f - k)) \leq Cs(f)$
\end{center}
Hence $\frac{d}{dt} s(f) \leq  C_Ms(f)$. Since $s(f) = 0$ to start, it must remain 0. 
\end{proof}
Using the setting from the proof above, assume without loss of generality that $X$ is compact and keep the notation that $s(f)$ is the distance of $f \in \mathbb{R}^{n^2}$ from $X$ and let 
\begin{center}
$s(t) = \sup\limits_x s(f(x,t)) = \sup{\ell(f(x,t) - k)}$
\end{center}
where the latter sup is over all $x \in M$, all $k \in \partial X$, and all $\ell \in S_k X$. Since this set and $M$ are compact, we can use Lemma 3.5 from \cite{Hamilton86} again to see that
\begin{center}
$\frac{d}{dt} s(t) \leq \sup \frac{d}{dt} \ell(f(x,t) - k)$
\end{center}
where the sup is over all $x,k,\ell$ as above with $\ell(f(x,t) - k ) = s(t)$. Then let $x $ be some point in $M$ where $f(x,t)$ is furthest away from $X$, $k$ is the unique closest point in $X$ to $f(x,t)$, and $\ell$ is the linear function of length 1 with gradient in the direction from $k$ to $f(x,t)$. Now
\begin{center}
$\frac{d}{dt} \ell(f(x,t) - k) = \ell(\Delta f) + \ell(|A|^2 f)$ 
\end{center}
Since $\ell(f(x,t))$ has its maximum at $x$, the term $\ell( \Delta f) = \Delta \ell(d) \leq 0$. Now we note that if $k \in X$, i.e.\ is positive semidefinite then $|A|^2k$ is too; hence from the lemma $\ell(|A|^2 k) \leq 0$. Now, suppose that $|A|^2 < C$ for all $t \in [0, T]$, then we would have from:
\begin{center}
$s(t) = \ell(|A|^2 f) \leq ||A|^2f - |A|^2 k| \leq C|f - k| = Cs(t)$
\end{center}
that for any time $t_0 \in [0, T]$ that $s(t) \leq s(0)e^{Ct}$. 
$\medskip$

Denoting by $S_n$ the shape tensor for $M_n$, we can apply the above reasoning (again, with some easy modifications/generalizations as discussed in section 4 of \cite{Hamilton86}) to see that since $k_{min}^n \to 0$, we have that $s_n(0) \to 0$ and in the same manner, if for $t_0$ (as above) $s_n(t_0) \to 0$, then $k_{min}^n(t_0) \to 0$ as well. 
$\medskip$

Recall that an offending sequence $M_n$ is one where $k_{\text{min}}$ goes to zero, but not all flows shrink smoothly to round spheres. Now for each of the $M_n$ recall we have $|A_n|^2 \leq \overline{C}$ for a universal constant $\overline{C}$ that works for any $t \in [0,T]$, so that $s_n(t_0) \leq s_n(0) e^{\overline{C}t_0}$ for all $n$. Since $s_n(0) \to 0$, we must indeed also have $s_n(t_0) \to 0$, so that the sequence $M_n(t_0)$ is still offending. 
$\medskip$

From the interior estimates then, again using a standard Arzela-Ascoli argument, we can extract a subsequence $M_\ell$ of $M_n$ that converges smoothly to an immersed manifold $L$ with positive semidefinite shape operator. By Hamilton's strong maximum principle since $L$ is compact and $k_{min} \geq 0$, for any time $t > 0$ for which the flow is defined we have $k_{min} > 0$, strictly. So pick a time $t < t_1 < T$, and set $\delta = t_1 - t$. Then $M_\ell(t_1) \to L(\delta)$ smoothly by continuous dependence and $k_{min}^L(\delta) = z > 0$. 
$\medskip$

Since $M_n(t_1) \to L(\delta)$ smoothly, $k_{min}^n(t_1) \to k_{min}^L(\delta)$, so that for large $n$ $k_{min}^n(t_1) > z/2$. Hence by Huisken's theorem, these will all proceed to shrink down to spheres, contradicting our assumption for $M_n$.

\end{appendix}

\bibliographystyle{amsplain}
\bibliography{bibliography}

\end{document}